\documentclass[11pt,a4paper]{article}

\usepackage{lmodern}
\usepackage[margin=1in]{geometry}
\usepackage{microtype}
\usepackage{amsmath,amssymb,amsthm,mathtools}
\usepackage{mathrsfs}
\usepackage{enumitem}
\usepackage{tikz}
\usetikzlibrary{arrows.meta,calc,matrix,positioning}
\usepackage[colorlinks,linkcolor=blue,citecolor=blue,urlcolor=blue]{hyperref}

\setlist{nosep}
\allowdisplaybreaks

\newtheorem{thm}{Theorem}[section]
\newtheorem{lem}[thm]{Lemma}
\newtheorem{prop}[thm]{Proposition}

\newtheorem{cor}[thm]{Corollary}
\theoremstyle{definition}
\newtheorem{defn}[thm]{Definition}

\theoremstyle{remark}
\newtheorem{rem}[thm]{Remark}

\DeclareMathOperator{\rk}{rk}
\DeclareMathOperator{\Ind}{Ind}
\DeclareMathOperator{\Aut}{Aut}
\newcommand{\Sn}{\mathfrak{S}_n}
\newcommand{\Sg}[1]{\mathfrak{S}_{#1}}
\newcommand{\Bn}{\mathcal{B}_n}
\newcommand{\B}[1]{\mathcal{B}_{#1}}
\newcommand{\VRep}{\operatorname{VRep}}
\newcommand{\grVRep}{\operatorname{grVRep}}
\newcommand{\one}{\mathbf{1}}
\newcommand{\cL}{\mathcal{L}}
\newcommand{\M}{\mathcal{M}}
\newcommand{\OS}{\mathrm{OS}}
\newcommand{\ch}{\operatorname{ch}}
\newcommand{\chB}{\operatorname{ch}_B}
\newcommand{\thmcite}[2]{\cite[#2]{#1}}

\newcommand{\Sym}{\Lambda}

\begin{document}

\begin{center}
{\large \bf Equivariant Kazhdan--Lusztig polynomials of thagomizer matroids \\with a hyperoctahedral group action}
\end{center}

\begin{center}
Matthew H. Y. Xie$^{1}$, Philip B. Zhang$^{2*}$, and Michael X. X. Zhong$^{1}$\\[6pt]

$^{1}$School of Mathematical Sciences,\\
Tianjin University of Technology, Tianjin 300384, P. R. China

$^{2}$College of Mathematical Sciences
\& Institute of Mathematics and Interdisciplinary Sciences,\\
Tianjin Normal University, Tianjin 300387, P. R. China\\[6pt]

Email addresses: \texttt{xie@email.tjut.edu.cn},
       \texttt{zhang@tjnu.edu.cn},
       \texttt{zhong.m@tjut.edu.cn}
\end{center}
\begingroup
\renewcommand{\thefootnote}{\fnsymbol{footnote}}
\footnotetext[1]{Corresponding author.}
\endgroup

\noindent\textbf{Abstract.}
We compute the equivariant Kazhdan--Lusztig and inverse Kazhdan--Lusztig polynomials of the thagomizer matroid $T_n=M(K_{1,1,n})$ under the action of the hyperoctahedral group $\Bn$. For $n\ge2$, this action identifies $\Bn$ with the full automorphism group of $T_n$. For every $n\ge0$, we give multiplicity-free decompositions of all coefficients of both polynomials into irreducible representations and prove that both coefficient sequences are strongly induced log-concave with respect to induction from $\Bn\times\Bn$ to $\B{2n}$. The formulas for Kazhdan--Lusztig polynomials follow from the palindromicity of the equivariant $Z$-polynomial and the known equivariant formulas for the graphic matroids of cycle graphs under the action of $\Sn$. The inverse formula follows from Proudfoot's equivariant inversion identity. For $n\ge2$, we also obtain explicit formulas for both polynomials of $M(K_{2,n})$ under the inherited action of $\Bn$, again with multiplicity-free coefficients.



\section{Introduction}

Elias, Proudfoot, and Wakefield~\cite{elias2016kazhdan} introduced the Kazhdan--Lusztig polynomial of a matroid, in analogy with the polynomials of Kazhdan and Lusztig~\cite{kazhdan1979representations} for Coxeter groups.
They defined, for each matroid $M$, a polynomial $P_M(t)\in\mathbb{Z}[t]$ characterized by a recursion over the lattice of flats together with a degree bound depending on the rank.
Although the coefficients of $P_M(t)$ were proved to be nonnegative by Braden, Huh, Matherne, Proudfoot, and Wang~\cite{braden2020singular}, explicit computations remain difficult, and structural properties such as unimodality or log-concavity are hard to establish in general.
Proudfoot, Xu, and Young~\cite{proudfoot2018zpolynomial} introduced the $Z$-polynomial of a matroid $M$, defined as a weighted sum over flats of the Kazhdan--Lusztig polynomials of the corresponding contractions, with weights determined by the ranks of the flats.
Braden and Vysogorets~\cite{braden2020deletion} showed that the degree bound and the palindromicity of the $Z$-polynomial give an alternative recursive characterization of $P_M(t)$.
Braden, Huh, Matherne, Proudfoot, and Wang~\cite{braden2020singular} later constructed an intersection cohomology module $\mathrm{IH}(M)$ and identified the associated $Z$-polynomial of $M$ with its Hilbert series.

This theory can be refined by equipping $M$ with an action of a finite group $W$ by matroid automorphisms.
Gedeon, Proudfoot, and Young~\cite{gedeon2017equivariant} defined the \emph{equivariant} Kazhdan--Lusztig polynomial~$P_M^W(t)$, whose coefficients lie in the representation ring of $W$ and whose graded dimension is $P_M(t)$. For further examples, see~\cite{proudfoot2019quniform}.
Gao and Xie~\cite{gao2021inverse} introduced the inverse Kazhdan--Lusztig polynomials and applied them to the computation of the Kazhdan--Lusztig polynomials of uniform matroids. Proudfoot~\cite{proudfoot2021incidence} subsequently generalized this notion by developing the Kazhdan--Lusztig--Stanley theory in an equivariant framework. Gao, Xie, and Yang~\cite{gao2022equivariant} computed the equivariant inverse Kazhdan--Lusztig polynomials of uniform matroids.
The coefficients of both equivariant polynomials are known to be honest representations~\cite[Theorem~1.3(1) and the discussion following Theorem~1.5]{braden2020singular}.

The thagomizer matroid $T_n$ is the graphic matroid of the complete tripartite graph $K_{1,1,n}$. Gedeon~\cite{gedeon2017thagomizer} conjectured a formula for $P_{T_n}^{\Sn}(t)$ under the action of $\Sn$.
Xie and Zhang~\cite{xie2019thagomizer} subsequently proved this formula.
However, for $n \ge 2$, the full automorphism group of $T_n$ is the larger hyperoctahedral group $\Bn \cong C_2 \wr \Sn$.
We determine both the equivariant Kazhdan--Lusztig polynomial $P_{T_n}^{\Bn}(t)$ and the equivariant inverse Kazhdan--Lusztig polynomial $Q_{T_n}^{\Bn}(t)$ under the action of the hyperoctahedral group, giving explicit multiplicity-free decompositions of their coefficients.

Write $A$ and $B$ for the vertices in the two singleton parts of $K_{1,1,n}$, and label the remaining vertices by $1,\ldots,n$.
We call $AB$ the spine and each pair $\{Ai,Bi\}$ a spike.
The group $\Bn$ acts by permuting the spikes and independently swapping the two edges within each spike, while fixing the spine.

Our first main result gives an explicit formula for $P_{T_n}^{\Bn}(t)$.
The irreducible representations $V_{\lambda,\mu}$ of $\Bn$ are indexed by bipartitions $(\lambda,\mu)$ of $n$, that is, ordered pairs of partitions with $|\lambda|+|\mu|=n$.
The first and second components correspond to the trivial and nontrivial characters of $C_2$, respectively.

\begin{thm}\label{thm:main}

For each $n\ge0$, we have

\[
P_{T_n}^{\Bn}(t)
=
V_{(n),\varnothing}
\;+\;
\sum_{k=1}^{\left\lfloor \frac{n}{2}\right\rfloor}
\left(\sum_{i=2k}^{n} V_{(n-i),(i-2k+2,2^{k-1})}\right)t^{k}
\]
in $\grVRep(\Bn)$, with the convention that $(0)=\varnothing$.
\end{thm}

Instead of using the defining recursion directly, we determine $P_{T_n}^{\Bn}(t)$ via the equivariant $Z$-polynomial. Palindromicity of the equivariant $Z$-polynomial is a general result of Proudfoot~\cite[Corollary~4.5]{proudfoot2021incidence}. We use this property to verify the formula in Theorem~\ref{thm:main} by comparison with the known equivariant Kazhdan--Lusztig polynomials of the graphic matroids of cycle graphs under the natural actions of the corresponding symmetric groups~\cite{proudfoot2016intersection}.

Our second main result concerns the inverse polynomial $Q_{T_n}^{\Bn}(t)$. After restriction to the standard subgroup $\Sn\subseteq\Bn$ and application of Pieri's rule, the formula recovers the decomposition obtained by Gao, Li, and Xie~\cite{gao2025thagomizer}. We prove the formula using Proudfoot's equivariant Kazhdan--Lusztig--Stanley inversion identity.

\begin{thm}\label{thm:inverse}
For each $n\ge0$, we have
\[
Q_{T_n}^{\Bn}(t)
=
\sum_{k=0}^{\left\lfloor\frac{n}{2}\right\rfloor}
\left(\sum_{i=2k}^{n} V_{(1^{n-i}),(2^k,1^{i-2k})}\right)t^k
\]
in $\grVRep(\Bn)$.
\end{thm}


The decompositions in Theorems~\ref{thm:main} and~\ref{thm:inverse} make the general positivity result explicit for this family. Taking dimensions recovers the nonequivariant polynomials computed by Gedeon~\cite{gedeon2017thagomizer} and by Gao, Li, and Xie~\cite{gao2025thagomizer}. Equivalently, the numerical coefficients count standard Young bitableaux of the shapes in the corresponding decomposition.

Matherne, Miyata, Proudfoot, and Ramos~\cite{matherne2023equivariant} studied equivariant log-concavity and its connections with representation stability.
Gao, Li, Xie, Yang, and Zhang~\cite{gao2026induced} introduced the notion of induced log-concavity for sequences of representations, extending the framework for equivariant log-concavity developed by Gedeon, Proudfoot, and Young~\cite{gedeon2017equivariant}, and established this property for the equivariant Kazhdan--Lusztig polynomials of uniform matroids.
Our third main result establishes strong induced log-concavity for the coefficient sequences of $P_{T_n}^{\Bn}(t)$ and $Q_{T_n}^{\Bn}(t)$.

\begin{thm}\label{thm:strong-induced-logconcavity}
For every $n\ge0$, the coefficient sequences of $P_{T_n}^{\Bn}(t)$ and
$Q_{T_n}^{\Bn}(t)$ are strongly induced log-concave with respect to induction
from $\Bn\times\Bn$ to $\B{2n}$.
Explicitly, if $(V_k)_{k\ge0}$ is either coefficient sequence, with $V_k=0$
beyond the degree of the polynomial, then, for every $1\le i\le j$,
\[
\Ind_{\Bn\times\Bn}^{\B{2n}}\!\left(V_i\boxtimes V_j\right)
-\Ind_{\Bn\times\Bn}^{\B{2n}}\!\left(V_{i-1}\boxtimes V_{j+1}\right)
\]
is an honest representation of $\B{2n}$.
\end{thm}

At the nonequivariant level, Wu and Zhang~\cite{wu2023logconcavity} proved that $P_{T_n}(t)$ is log-concave, and Gao, Li, and Xie~\cite{gao2025thagomizer} proved the same property for $Q_{T_n}(t)$.
Taking dimensions in Theorem~\ref{thm:strong-induced-logconcavity} with $i=j$ recovers both log-concavity results.

Under the Frobenius characteristic for the wreath product $C_2\wr\Sg{2n}$, the inequalities in Theorem~\ref{thm:strong-induced-logconcavity} amount to Schur positivity of the corresponding differences in $\Sym[X,Y]$ for every $1\le i\le j$.
Pieri's rule reduces the problem to strong Schur log-concavity for two families of truncated Schur sums.
We prove the required inequalities for the first family using Speyer's criterion for Schur positivity~\cite[Corollary~1]{Speyer2026}.
For the second family, a factorization of skew Schur polynomials in two variables and two applications of Pieri's rule express the Schur coefficients of the products as sums of cardinalities of integer intervals.
Comparing these cardinalities term by term proves the required inequalities.

We also obtain formulas for the cycle matroid of $K_{2,n}$.
Deleting the spine $e_*=AB$ gives $T_n\setminus e_*=M(K_{2,n})$.
Since the action of $\Bn$ fixes $e_*$, the deletion inherits a $\Bn$-action.
Using the equivariant deletion formula of Ferroni, Matherne, and
Vecchi~\cite{ferroni2025deletion} and Proudfoot's equivariant
Kazhdan--Lusztig--Stanley theory~\cite{proudfoot2021incidence},
we deduce the following corollary of Theorems~\ref{thm:main}
and~\ref{thm:inverse}.

\begin{cor}\label{cor:spine-deletion}
For every $n\ge2$, with respect to the inherited action of $\Bn$, we have
\begin{equation}\label{eq:spine-deletion}
P_{M(K_{2,n})}^{\Bn}(t)=P_{T_n}^{\Bn}(t)+t\one_{\Bn},
\qquad
Z_{M(K_{2,n})}^{\Bn}(t)=Z_{T_n}^{\Bn}(t),
\end{equation}
and
\begin{equation}\label{eq:inverse-spine-deletion}
Q_{M(K_{2,n})}^{\Bn}(t)
=Q_{T_n}^{\Bn}(t)-V_{(1^n),\varnothing}
+tV_{(2,1^{n-2}),\varnothing}.
\end{equation}
Here $\one_{\Bn}$ denotes the trivial representation of $\Bn$.
Every irreducible constituent of each coefficient of $P_{M(K_{2,n})}^{\Bn}(t)$ and
$Q_{M(K_{2,n})}^{\Bn}(t)$ has multiplicity one.
\end{cor}


The rest of the paper is organized as follows. Section~\ref{sec:preliminaries} reviews background on Frobenius characteristics and equivariant matroid invariants.
In Section~\ref{sec:thagomizer}, we describe the action of $\Bn$ on $T_n$ and classify its orbits on flats.
Section~\ref{sec:cycle} computes the generating functions of the equivariant Kazhdan--Lusztig and $Z$-polynomials of the graphic matroids of cycle graphs.
Section~\ref{sec:z-comparison} derives $P_{T_n}^{\Bn}(t)$ by comparing the equivariant $Z$-polynomial of $T_n$ under the action of $\Bn$ with the corresponding terms coming from the matroids $C_k$.
In Section~\ref{sec:inverse}, we derive $Q_{T_n}^{\Bn}(t)$ using Proudfoot's inversion identity.
We also prove Corollary~\ref{cor:spine-deletion} there and show that $\Bn$ is the full automorphism group of $M(K_{2,n})$ for $n\ge3$.
Section~\ref{sec:ilc} proves strong induced log-concavity for the coefficient sequences of $P_{T_n}^{\Bn}(t)$ and $Q_{T_n}^{\Bn}(t)$.
Finally, Appendix~\ref{app:charpoly} gives the equivariant characteristic polynomial of $T_n$ under the action of $\Bn$ and proves strong induced log-concavity for the graded pieces of its Orlik--Solomon algebra.

\section{Preliminaries}\label{sec:preliminaries}

\subsection{Frobenius characteristic and plethystic notation}\label{sec:frob-pleth}

Let $\Sym$ be the ring of symmetric functions over $\mathbb{Z}$. We write $\Sym_n$ for its homogeneous component of degree~$n$.
Let $X=\{x_1,x_2,x_3,\ldots\}$ and $Y=\{y_1,y_2,y_3,\ldots\}$ be two independent alphabets.
For an alphabet $A$, let $\Sym[A]$ denote the ring of symmetric functions in $A$ over $\mathbb{Z}$. We set $\Sym[X,Y]:=\Sym[X]\otimes_{\mathbb{Z}}\Sym[Y]$, whose elements are symmetric separately in $X$ and $Y$.
We write $\Sym[X,Y]_n$ for the component of degree~$n$ of $\Sym[X,Y]$ with respect to total degree.
The subscript is often suppressed when the degree is clear from context.
We write $h_n$ and $e_n$ for the complete homogeneous and elementary symmetric functions, respectively. For $n\ge 1$,
\[
h_n[X]=\sum_{i_1\le\cdots\le i_n}x_{i_1}\cdots x_{i_n},
\qquad
e_n[X]=\sum_{i_1<\cdots<i_n}x_{i_1}\cdots x_{i_n},
\]
where the indices are positive integers, and $h_0=e_0=1$.

	We follow Haglund's notation for plethystic substitution~\cite{haglund2008q}.
	Let $E=E(t_1,t_2,t_3,\ldots)$ be a formal series of rational functions in commuting parameters.
	For $k\ge 1$, the plethystic substitution of $E$ into the power sum $p_k$ is defined by
	\[
	p_k[E]:=E(t_1^k,t_2^k,\ldots).
	\]
	For $f\in\Sym$, write $f=\sum_{\lambda} c_{\lambda}p_{\lambda}$ in the power-sum basis of $\Sym\otimes_{\mathbb Z}\mathbb Q$, where $c_{\lambda}\in\mathbb Q$, and define
	\[
	f[E]:=\sum_{\lambda} c_{\lambda}\prod_i p_{\lambda_i}[E].
	\]
	The notations $h_n[E]$ and $e_n[E]$ denote this substitution applied to $h_n$ and $e_n$, respectively.
	In particular, writing $X=x_1+x_2+\cdots$, we have $p_k[X]=\sum_i x_i^k$ and hence $f[X]=f(x_1,x_2,\ldots)$ for any $f\in\Sym$.
	Moreover, $p_k[-X]=-p_k[X]$.

	\begin{lem}[\thmcite{haglund2008q}{Theorem~1.27}]\label{lem:haglund-1}
		Let $E=E(t_1,t_2,t_3,\ldots)$ and $F=F(w_1,w_2,w_3,\ldots)$ be formal series of rational functions in commuting parameters.
		Then for all $n\ge 0$,
		\begin{align*}
		h_{n}[E+F]&=\sum_{j=0}^{n}h_{j}[E]h_{n-j}[F],\\
		h_{n}[E-F]&=\sum_{j=0}^{n}(-1)^{n-j}h_{j}[E]e_{n-j}[F].
		\end{align*}
	\end{lem}

For an alphabet $A$, write $H_A(u):=\sum_{n\ge 0} h_n[A]u^n$.
Specializing Lemma~\ref{lem:haglund-1} to $E=X$ and $F=Y$ gives
\begin{equation}\label{eq:HXHY}
H_{X+Y}(u)=H_X(u)H_Y(u).
\end{equation}

Let $W$ be a finite group. We write $\VRep(W)$ for its complex representation ring, with addition induced by direct sum and multiplication by tensor product. All representations are assumed to have finite dimension. We will use the usual Frobenius characteristic for $\Sn$ and the Frobenius characteristic for the wreath product $\Bn=C_2\wr\Sn$.

	For each $n$, let $\ch:\VRep(\Sn)\to \Sym_n$ be the usual Frobenius characteristic~\cite[\S I.7]{macdonald1995symmetric}, characterized by the property that
\[
\ch(V_\lambda)=s_\lambda,
\]
where $V_\lambda$ denotes the irreducible representation of $\Sn$ indexed by the partition $\lambda$ of $n$, and $s_\lambda$ denotes the corresponding Schur function in $\Sym_n$.
It satisfies $\ch(\one_{\Sn})=h_n$.
For $a,b\ge0$ and representations $U\in\VRep(\Sg{a})$ and $V\in\VRep(\Sg{b})$, we have
\begin{equation}\label{eq:frob-sn-ind}
\ch\!\left(\Ind_{\Sg{a}\times \Sg{b}}^{\Sn}(U\boxtimes V)\right)=\ch(U)\,\ch(V)
\qquad(a+b=n).
\end{equation}

Write $\Bn=C_2\wr \Sn$ for the hyperoctahedral group.
The irreducible complex representations of~$\Bn$ are indexed by bipartitions $(\lambda,\mu)$ of $n$ and denoted by $V_{\lambda,\mu}$.
	Following Macdonald~\cite[Chapter~I, Appendix~B]{macdonald1995symmetric} and the treatments in~\cite{macdonald1980polynomial, mendes2004lambdaring}, we use the Frobenius characteristic for the wreath product $C_2\wr\Sn$ (also called the $\lambda$-ring Frobenius characteristic), namely the following isomorphism of abelian groups
\[
\chB:\VRep(\Bn)\longrightarrow \Sym[X,Y]_n,
\]
defined on irreducibles by
	\begin{align}\label{def:chB}
	\chB(V_{\lambda,\mu})=s_\lambda[X]\,s_\mu[Y].
	\end{align}
	See, for example, Athanasiadis~\cite[\S2]{athanasiadis2020some} in the setting of equivariant $\gamma$-positivity.
	In particular, $\chB(\one_{\B{n}})=h_n[X]$.
	For $a,b\ge0$ and representations $U\in\VRep(\B{a})$ and $V\in\VRep(\B{b})$, the compatibility of $\chB$ with induction from the standard block subgroups $\B{a}\times\B{b}\subseteq\Bn$ gives
	\begin{equation}\label{eq:frob-bn-ind}
	  \chB\!\left(\Ind_{\B{a}\times \B{b}}^{\B{n}}(U\boxtimes V)\right)
  = \chB(U)\,\chB(V)
\end{equation}
for $a+b=n$. Embed $\Sn$ in $\Bn$ as the signed permutations with all signs positive.
For $V\in\VRep(\Sn)$, we also have
\begin{equation}\label{eq:frob-sn-to-bn}
  \chB\!\left(\Ind_{\Sn}^{\Bn} V\right)
  = \ch(V)[X+Y].
\end{equation}

For the reader's convenience, we include a proof of the following standard lemma.

\begin{lem}\label{lem:ind-sk-bn-k}
Let $a,b\ge 0$ with $a+b=n$.
Let $\B{a}\times \B{b}$ be the standard block subgroup of $\Bn$, acting on the first $a$ and last $b$ coordinates, respectively, and embed $\Sg{a}$ in $\B{a}$ as the signed permutations with all signs positive.
These choices specify the subgroup $\Sg{a}\times \B{b}\subseteq\Bn$. Conjugating either of these standard embeddings yields a subgroup conjugate to $\Sg{a}\times\B{b}$ in $\Bn$.
Then for $U\in \VRep(\Sg{a})$ and $V\in \VRep(\B{b})$
\[
	\chB\!\left(\Ind_{\Sg{a}\times \B{b}}^{\Bn}(U\boxtimes V)\right)
	=
	\ch(U)[X+Y]\cdot \chB(V).
	\]
	\end{lem}
\begin{proof}
For the chosen standard embeddings
\(
\Sg{a}\times \B{b}\subseteq \B{a}\times \B{b}\subseteq \Bn,
\)
transitivity of induction gives
\[
\Ind_{\Sg{a}\times \B{b}}^{\Bn}(U\boxtimes V)
\cong
\Ind_{\B{a}\times \B{b}}^{\Bn}\!\Bigl(\Ind_{\Sg{a}}^{\B{a}}U\boxtimes V\Bigr).
\]
Applying $\chB$ and using~\eqref{eq:frob-bn-ind}, we deduce that
\[
\chB\!\left(\Ind_{\Sg{a}\times \B{b}}^{\Bn}(U\boxtimes V)\right)
=
\chB\!\left(\Ind_{\Sg{a}}^{\B{a}}U\right)\cdot \chB(V).
\]
Finally, applying~\eqref{eq:frob-sn-to-bn} with $n=a$, we conclude that
\(
\chB\!\left(\Ind_{\Sg{a}}^{\B{a}}U\right)=\ch(U)[X+Y].
\)

If a conjugate subgroup is used instead, the induced representation has the same isomorphism class. More explicitly, for $H\le G$ and $g\in G$, let ${}^gU$ be the representation of $gHg^{-1}$ in which $ghg^{-1}$ acts as $h$ acts on $U$. Then
\[
\Ind_{gHg^{-1}}^G({}^gU)\cong\Ind_H^G U.
\]
\end{proof}

\subsection{Equivariant Kazhdan--Lusztig and \texorpdfstring{$Z$}{Z}-polynomials}\label{sec:ekl-z}

Throughout this paper, we work with loopless matroids. Let $W$ be a finite group acting by automorphisms on a matroid~$\M$, and write $\grVRep(W)=\VRep(W)[t]$.

For a subset $S$ of the ground set of a matroid $\M$, we write $\M|_S$ for the restriction of $\M$ to $S$ and $\M/S$ for the contraction of $\M$ by $S$.
We use the following simultaneous characterization of the equivariant Kazhdan--Lusztig and $Z$-polynomials, which follows from Braden, Huh, Matherne, Proudfoot, and Wang~\cite[Corollary~A.5]{braden2020singular}. We state it in the form given by Ferroni, Matherne, and Vecchi~\cite[Theorem-Definition~2.1]{ferroni2025deletion}.

\begin{defn}[\thmcite{ferroni2025deletion}{Theorem-Definition 2.1}]\label{def:equivariant-kl-z-polynomials}
There is a unique way of associating to each loopless matroid $\M$ and each group action $W\curvearrowright \M$ two graded virtual representations $P_\M^W(t), Z_\M^W(t) \in \grVRep(W)$ in such a way that
\begin{enumerate}[label={(\roman*)},font=\normalfont]
    \item If $\rk(\M) = 0$, then $P_\M^W(t) = Z_\M^W(t) = \one_W$.
    \item If $\rk(\M) > 0$, then $\deg P_\M^W(t) < \frac{1}{2}\rk(\M)$.
    \item For every $\M$,
    \[
    Z_\M^W(t) \;=\; \sum_{[F] \in \cL(\M)/W} t^{\rk(F)} \Ind_{W_F}^W P_{\M/F}^{W_F}(t)
    \]
    is palindromic and has degree $\rk(\M)$.
\end{enumerate}
Here $\one_W$ is the trivial representation of $W$, $\cL(\M)/W$ denotes the set of $W$-orbits of flats, and $W_F$ is the stabilizer of the flat $F$.
\end{defn}

Definition~\ref{def:equivariant-kl-z-polynomials} is formulated entirely in terms of the lattice of flats together with the induced actions on contractions.
In particular, the equivariant Kazhdan--Lusztig and $Z$-polynomials depend only on the action of $W$ on $\cL(\M)$.
The kernel of this action acts trivially on every coefficient of both polynomials, as follows by induction on the rank and the uniqueness in Definition~\ref{def:equivariant-kl-z-polynomials}.

The following lemma on the equivariant Kazhdan--Lusztig polynomial in the Boolean case is a direct corollary of the corresponding result for ordinary Kazhdan--Lusztig polynomials and equivariant positivity.

\begin{lem}\label{lem:boolean}
Let $\M$ be a loopless matroid of rank $r$ whose lattice of flats~$\cL(\M)$ is a Boolean lattice.
Then for any action $W\curvearrowright \M$ we have
\[
P_\M^W(t)=\one_W\in \VRep(W)\subseteq \grVRep(W).
\]
\end{lem}
\begin{proof}
The ordinary Kazhdan--Lusztig polynomial depends only on the lattice of flats, so $P_\M(t)=1$ by Elias, Proudfoot, and Wakefield~\cite[Corollary~2.10]{elias2016kazhdan}.
Write $P_\M^W(t)=\sum_{i\ge0}[V_i]t^i$, where each $V_i$ is an honest representation of $W$ by~\cite[Theorem~1.3(1)]{braden2020singular}.
Taking dimensions gives $\sum_{i\ge0}(\dim V_i)t^i=P_\M(t)=1$, so $V_i=0$ for all $i>0$.
Finally, $[V_0]=\one_W$ by Gedeon, Proudfoot, and Young~\cite[Corollary~2.10]{gedeon2017equivariant}.
\end{proof}

\medskip

For comparison, we recall the equivariant Kazhdan--Lusztig polynomial of $T_n$ under the action of $\Sn$ computed by Xie and Zhang~\cite{xie2019thagomizer}. This result is not used in the proof of Theorem~\ref{thm:main}.
For $k\ge2$, let $C_k$ denote the graphic matroid of a cycle graph with $k$ edges, equipped with the natural action of $\Sg{k}$. Note that $C_k\cong U_{k-1,k}$.
For $k=2$, the cycle consists of two parallel edges.

\begin{prop}[\thmcite{xie2019thagomizer}{Theorem~2.1}]\label{prop:xz}
For $n\ge 0$, write $\mathsf{P}_n^{A}(X;t):=\ch\!\left(P_{T_n}^{\Sn}(t)\right)[X]\in \Sym[X][t]$ and $\mathsf{C}_k(X;t):=\ch\!\left(P_{C_k}^{\Sg{k}}(t)\right)[X]$ for $k\ge 2$.
Then
\[
\mathsf{P}_n^{A}(X;t)=h_n[X]+t\sum_{k=2}^{n} h_{n-k}[X]\,\mathsf{C}_k(X;t).
\]
\end{prop}
For $n=0$, the formula follows from $P_{T_0}^{\Sg{0}}(t)=\one_{\Sg{0}}$, with the convention that an empty sum is zero.

\begin{rem}
Restriction to the standard subgroup $\Sn\subseteq\Bn$ corresponds, under the Frobenius characteristic, to identifying the alphabets $X$ and $Y$.
Thus~\eqref{eq:main} below specializes to Proposition~\ref{prop:xz}.
For the inverse polynomial, the same specialization, followed by Pieri's rule, recovers the irreducible decomposition in~\cite{gao2025thagomizer}.

The Type A formula in Proposition~\ref{prop:xz} was first discovered through a serendipitous observation.
In contrast, the Type B formula in the present work was conjectured directly from computations in SageMath.
The SageMath code used in our computations is available at \url{https://github.com/mathxie/equivariant-KLS-polynomials}.
\end{rem}

\section{The thagomizer matroid under the action of \texorpdfstring{$\Bn$}{Bn}}\label{sec:thagomizer}

\subsection{The automorphism group}\label{sec:aut}

Fix $n\ge 0$.
Write $K_{1,1,n}$ for the graph with vertex set $\{A,B,1,\dots,n\}$ and edge set
\[
E_n=\{e_*\}\cup\{a_i,b_i \colon 1\le i\le n\},
\qquad
e_*:=AB,\quad a_i:=A i,\quad b_i:=B i.
\]
Thus $T_n=M(K_{1,1,n})$ is a graphic matroid of rank $n+1$ on $2n+1$ elements.
We refer to the pair~$\{a_i,b_i\}$ as the $i$th \emph{spike} and to $e_*$ as the \emph{spine}.

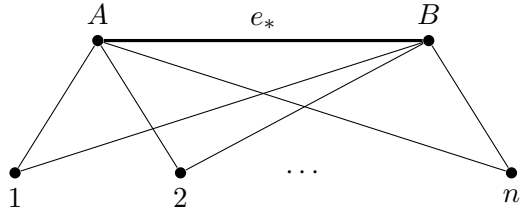
\begin{figure}[ht]
\centering
\begin{tikzpicture}[scale=1.1]
  \node[circle, fill=black, inner sep=1.5pt, label=above:$A$] (A) at (-2,1.6) {};
  \node[circle, fill=black, inner sep=1.5pt, label=above:$B$] (B) at (2,1.6) {};
  \draw[very thick] (A) -- node[above] {$e_*$} (B);

  \node[circle, fill=black, inner sep=1.5pt, label=below:$1$] (v1) at (-3,0) {};
  \node[circle, fill=black, inner sep=1.5pt, label=below:$2$] (v2) at (-1,0) {};
  \node (vd) at (0.5,0) {$\cdots$};
  \node[circle, fill=black, inner sep=1.5pt, label=below:$n$] (vn) at (3,0) {};

  \draw (A) -- node[above, sloped, font=\small, pos=0.55] {} (v1);
  \draw (B) -- node[above, sloped, font=\small, pos=0.55] {} (v1);

  \draw[above, sloped, font=\small, pos=0.55] (A) -- (v2);
  \draw[above, sloped, font=\small, pos=0.55] (B) -- (v2);

  \draw (A) -- node[above, sloped, font=\small, pos=0.55] {} (vn);
  \draw (B) -- node[above, sloped, font=\small, pos=0.55] {} (vn);
\end{tikzpicture}
\caption{The graph $K_{1,1,n}$. The spine edge is $e_*=AB$, and the $i$th spike is the pair of edges $\{a_i,b_i\}$ joining the vertex $i$ to $A$ and $B$.}\label{fig:thagomizer}
\end{figure}

We begin by identifying the full matroid automorphism group of $T_n$.

\begin{prop}\label{prop:aut}
Assume $n\ge 2$.
The full matroid automorphism group of $T_n$ is the wreath product
\[
\Aut(T_n)\cong \Bn=C_2\wr \Sn\cong C_2^n\rtimes \Sn,
\]
generated by permutations of the $n$ spikes $\{a_i,b_i\}$ and by independent transpositions $a_i\leftrightarrow b_i$ for each~$i$.
\end{prop}

\begin{proof}
The three-element circuits are precisely $\{e_*,a_i,b_i\}$ for $1\le i\le n$.
Since $n\ge2$, the spine is the unique element contained in more than one three-element circuit.
Every automorphism therefore fixes $e_*$ and permutes the pairs $\{a_i,b_i\}$.
Conversely, the only other circuits are $\{a_i,b_i,a_j,b_j\}$ with $i\ne j$, so every permutation of the spikes and every independent swap within a spike preserves all circuits.
\end{proof}


\begin{rem}
For $n=0$, $T_0$ has one element and $\Aut(T_0)$ is trivial.
For $n=1$, $T_1$ is the uniform matroid of rank $2$ on $3$ elements, so $\Aut(T_1)\cong \mathfrak{S}_3$, while the subgroup $\B{1}=C_2\wr \Sg{1}\cong C_2$ is generated by the transposition $(a_1\,b_1)$, which fixes $e_*$.
\end{rem}

\subsection{Flats of \texorpdfstring{$T_n$}{Tn} and their orbits under the action of \texorpdfstring{$\Bn$}{Bn}}\label{sec:flats}

The flats of $T_n$ fall into two families, distinguished by whether they contain the spine $e_*$. See Gedeon~\cite[\S3]{gedeon2017thagomizer}.

\begin{lem}\label{lem:flats}
Let $F$ be a flat of $T_n$.
\begin{enumerate}[label={(\alph*)}]
    \item If $e_*\notin F$, then for each spike $\{a_i,b_i\}$ the intersection $F\cap\{a_i,b_i\}$ is one of $\varnothing,\{a_i\},\{b_i\}$.
    Writing $I:=\{i\colon a_i\in F\}$ and $J:=\{i\colon b_i\in F\}$, we have $I\cap J=\varnothing$ and $\rk(F)=|I|+|J|$.
    \item If $e_*\in F$, then there exists a subset $S\subseteq\{1,\dots,n\}$ such that
    \[
    F=\{e_*\}\cup\{a_i,b_i\colon i\in S\}.
    \]
    In this case $\rk(F)=|S|+1$.
\end{enumerate}
Conversely, every edge set described in \textup{(a)} or \textup{(b)} is a flat of $T_n$.
\end{lem}
\begin{proof}
An edge set is a flat of a graphic matroid if and only if every edge outside it joins two distinct connected components of the corresponding spanning subgraph.
If $e_*\notin F$, the flat $F$ cannot contain both $a_i$ and $b_i$, since they would connect $A$ to $B$ and force $e_*$ into the closure of $F$.
If $e_*\in F$, containing either $a_i$ or $b_i$ forces the other into the closure, so both must belong to $F$.
These observations give the two necessary forms.

Conversely, for disjoint $I,J$, the edge set in \textup{(a)} consists of a star centered at $A$ with leaves in $I$, a star centered at $B$ with leaves in $J$, and isolated vertices.
Every omitted edge joins distinct components, so this edge set is a flat. It is a forest and has rank $|I|+|J|$.
For the edge set in \textup{(b)}, the vertices $A,B$ and those indexed by $S$ form one connected component, while all other vertices are isolated.
Again every omitted edge joins distinct components, and the rank is $(|S|+2)-1=|S|+1$.
\end{proof}

To evaluate the sum over orbits of flats defining the $Z$-polynomial in Section~\ref{sec:z-comparison}, we also need explicit orbit representatives and stabilizers.

\begin{lem}\label{lem:stabilizers}
Let $F$ be a flat of $T_n$.
\begin{enumerate}[label={(\alph*)}]
    \item If $e_*\notin F$ and $k:=\rk(F)$, then, under the action of $\Bn$, $F$ lies in the orbit of
    \[
    F^{\mathrm{I}}_{k}:=\{a_1,\dots,a_k\},
    \]
    and the stabilizer of this representative satisfies
    \[
    (\Bn)_{F^{\mathrm{I}}_{k}}\cong \Sg{k}\times \B{n-k}.
    \]
    \item If $e_*\in F$ and $k:=\rk(F)-1$, then, under the action of $\Bn$, $F$ lies in the orbit of
    \[
    F^{\mathrm{II}}_{k}:=\{e_*\}\cup\{a_1,b_1,\dots,a_k,b_k\},
    \]
    and the stabilizer of this representative satisfies
    \[
    (\Bn)_{F^{\mathrm{II}}_{k}}\cong \B{k}\times \B{n-k}.
    \]
\end{enumerate}
\end{lem}

\begin{proof}
(a) Let $F$ be a Type~(a) flat. Then $F$ is determined by disjoint sets $I,J$ as in Lemma~\ref{lem:flats}(a) with $|I|+|J|=k$.
Applying the swaps of $a_i$ and $b_i$ for all $i\in J$ sends $F$ to $\{a_i\colon i\in I\cup J\}$.
Then apply a permutation in $\Sn\subseteq \Bn$ that maps~$I\cup J$ to $\{1,\dots,k\}$.
This shows that every Type~(a) flat of rank $k$ lies in the orbit of $F_k^{\mathrm{I}}$.

To determine the stabilizer of $F_k^{\mathrm{I}}$, write an element of $\Bn$ as $(\varepsilon,\sigma)\in C_2^n\rtimes \Sn$, where $\varepsilon$ records the independent swaps and $\sigma$ permutes spike indices.
Stabilizing $F_k^{\mathrm{I}}$ forces $\sigma$ to preserve the index set $\{1,\dots,k\}$.
Moreover, if $i\le k$, then $a_i\in F_k^{\mathrm{I}}$ and $b_i\notin F_k^{\mathrm{I}}$, so a stabilizer element cannot swap $a_i$ with $b_i$.
Hence $\varepsilon_i$ must be trivial for $i\le k$, and the action on $\{1,\dots,k\}$ reduces to the permutation subgroup $\Sg{k}$.
On indices $\{k+1,\dots,n\}$, we may freely permute indices and swap $a_i\leftrightarrow b_i$ independently.
This yields a factor $\B{n-k}$ and proves $(\Bn)_{F_k^{\mathrm{I}}}\cong \Sg{k}\times \B{n-k}$.

(b) Let $F$ be a Type~(b) flat determined by $S\subseteq\{1,\dots,n\}$ with $|S|=k$.
Apply a permutation in $\Sn\subseteq \Bn$ sending $S$ to $\{1,\dots,k\}$ to obtain the representative $F_k^{\mathrm{II}}$.
Inside the first $k$ spikes, swaps $a_i\leftrightarrow b_i$ preserve $F_k^{\mathrm{II}}$ because both edges are present.
Similarly, inside the remaining $n-k$ spikes, permutations and swaps preserve $F_k^{\mathrm{II}}$ because no edges from those spikes are present.
This gives a stabilizer isomorphic to $\B{k}\times \B{n-k}$.
\end{proof}

\begin{rem}
Compare Lemma~\ref{lem:stabilizers} with the symmetric group action~$\Sn \curvearrowright T_n$ studied in~\cite{xie2019thagomizer}.
For a Type~(a) flat $F$, let $i$ and $j$ be the numbers of edges of the forms $a_\ell$ and $b_\ell$ in $F$, respectively, with $i+j=k$.
Under the action of $\Sn$, the orbit of $F$ depends on the pair $(i,j)$, not just on $k = \rk(F)$; the stabilizer of $F$ is isomorphic to $\Sg{i} \times \Sg{j} \times \Sg{n-k}$.
Under the action of the larger group $\Bn$, all such flats for a fixed $k$ form a single orbit.
For a general flat $F$ of this type, the stabilizer is conjugate to the displayed subgroup $\Sg{k}\times\B{n-k}$.
Under this identification, the first factor permutes the $k$ spikes
meeting $F$; for mixed choices of $a$- and $b$-edges, these permutations
may be accompanied by swaps within spikes. The second factor acts on the
$n-k$ spikes disjoint from $F$, permuting them and independently swapping
the two edges within each spike.
\end{rem}

\section{Matroids of cycle graphs and their generating functions}\label{sec:cycle}

Recall that, for $n\ge 2$, $C_n\cong U_{n-1,n}$ is the graphic matroid of a cycle with $n$ edges, equipped with the action of $\Sn$ that permutes its ground set.
For $n\ge 2$, set
\[
\mathsf{C}_n(Y;t):=\ch\!\left(P_{C_n}^{\Sn}(t)\right)[Y]\in \Sym[Y][t]\subseteq \Sym[X,Y][t],
\]
and set $\mathsf{C}_0(Y;t)=\mathsf{C}_1(Y;t)=0$.

Proudfoot, Wakefield, and Young~\cite{proudfoot2016intersection} computed these polynomials explicitly and showed that each nonzero graded piece is an irreducible representation of $\Sn$.
We record their formula using the Frobenius characteristic map.

\begin{thm}[\thmcite{proudfoot2016intersection}{Theorem~1.2(2)}]\label{thm:cycle-input}
For $n\ge 2$,
\[
\mathsf{C}_n(Y;t)=\sum_{i=0}^{\left\lfloor \frac{n}{2}\right\rfloor-1} s_{(n-2i,2^i)}[Y]\,t^i.
\]
\end{thm}

\medskip

To express Theorem~\ref{thm:main} in terms of these cycle polynomials, write, for $n\ge0$,
\[
\mathsf{P}_n(X,Y;t):=\chB\!\left(P_{T_n}^{\Bn}(t)\right)\in \Sym[X,Y][t].
\]
By~\eqref{def:chB}, Theorem~\ref{thm:main} is equivalent to the following identity of symmetric functions
\begin{equation}\label{eq:main}
\mathsf{P}_n(X,Y;t)=h_n[X]+t\sum_{k=2}^{n} h_{n-k}[X]\,\mathsf{C}_k(Y;t).
\end{equation}

We use the following generating functions for the equivariant Kazhdan--Lusztig and $Z$-polynomials of the matroids $C_n$:
\begin{align*}
\Phi_C(t,u)
&:=\sum_{n\ge 2}\ch\!\left(P_{C_n}^{\Sn}(t)\right)[Y]\,u^{n-1}
=\sum_{n\ge 2}\mathsf{C}_n(Y;t)\,u^{n-1},\\
\mathcal{Z}_C(t,u)
&:=\sum_{n\ge 2}\ch\!\left(Z_{C_n}^{\Sn}(t)\right)[Y]\,u^{n-1}.
\end{align*}


\begin{prop}\label{prop:ZC-in-terms-of-PhiC}
We have
\begin{equation}\label{eq:ZC-PhiC}
\mathcal{Z}_C(t,u)=\frac{H_Y(tu)-1-h_1[Y]\cdot tu}{tu}+H_Y(tu)\,\Phi_C(t,u).
\end{equation}
\end{prop}
\begin{proof}
Fix $n\ge 2$ and apply Definition~\ref{def:equivariant-kl-z-polynomials}(iii) to $C_n$ with the action of $\Sn$.
The lattice of flats of $C_n$ consists of all subsets of size at most $n-2$ together with the full set $E$.
Thus the orbits of flats under the action of $\Sn$ are indexed by $k=|F|\in\{0,1,\dots,n-2\}$ together with the orbit $F=E$.

Fix $0\le k\le n-2$ and let $F\subseteq E$ be a flat of size $k$.
Then $\rk(F)=k$.
The stabilizer is $\Sg{k}\times \Sg{n-k}$, acting by permutations of $F$ and of its complement $E\setminus F$.
Since $C_n=U_{n-1,n}$, the contraction formula for uniform matroids gives $C_n/F=U_{n-k-1,n-k}\cong C_{n-k}$ on the remaining $n-k$ elements, with the natural action of $\Sg{n-k}$.
The $\Sg{k}$ factor acts trivially on this contraction, so the defining properties in Definition~\ref{def:equivariant-kl-z-polynomials} imply
\(
P_{C_n/F}^{\Sg{k}\times \Sg{n-k}}(t)=\one_{\Sg{k}}\boxtimes P_{C_{n-k}}^{\Sg{n-k}}(t).
\)
Inducing to $\Sn$ and applying \eqref{eq:frob-sn-ind} yields
\[
\ch\!\left(\Ind_{\Sg{k}\times \Sg{n-k}}^{\Sn}(\one_{\Sg{k}}\boxtimes P_{C_{n-k}}^{\Sg{n-k}}(t))\right)[Y]
=h_k[Y]\,\mathsf{C}_{n-k}(Y;t).
\]

For the remaining orbit $F=E$, we have $\rk(E)=n-1$ and $C_n/E$ has rank $0$, hence $P_{C_n/E}^{\Sn}(t)=\one_{\Sn}$ by Definition~\ref{def:equivariant-kl-z-polynomials}(i), and $\ch(\one_{\Sn})[Y]=h_n[Y]$.
Therefore
\begin{equation}\label{eq:cycle-Zn}
\ch\!\left(Z_{C_n}^{\Sn}(t)\right)[Y]
=t^{n-1}h_n[Y]+\sum_{k=0}^{n-2} t^k\,h_k[Y]\,\mathsf{C}_{n-k}(Y;t).
\end{equation}
Multiplying \eqref{eq:cycle-Zn} by $u^{n-1}$ and summing over $n\ge 2$ yields
\[
\mathcal{Z}_C(t,u)
=\sum_{n\ge 2} t^{n-1}h_n[Y]\,u^{n-1}
 +\sum_{n\ge 2}\sum_{k=0}^{n-2} t^k\,h_k[Y]\;\mathsf{C}_{n-k}(Y;t)\,u^{n-1}.
\]
The first sum is
\[
\sum_{n\ge 2} t^{n-1} h_n[Y]\,u^{n-1}
=\frac{H_Y(tu)-1-h_1[Y]\cdot tu}{tu},
\]
and the second is
\[
\sum_{n\ge 2}\sum_{k=0}^{n-2} t^k\,h_k[Y]\;\mathsf{C}_{n-k}(Y;t)\,u^{n-1}
=H_Y(tu)\,\Phi_C(t,u).
\]
Substituting these two expressions into the generating series gives~\eqref{eq:ZC-PhiC}.
\end{proof}

We also need the palindromicity of the $Z$-polynomials of the graphic matroids of cycle graphs.

\begin{lem}\label{lem:pal-ZC}
We have
\[
\mathcal{Z}_C(t,u)=\mathcal{Z}_C(t^{-1},tu).
\]
\end{lem}
\begin{proof}
For each $n\ge 2$, Definition~\ref{def:equivariant-kl-z-polynomials}(iii) implies that $Z_{C_n}^{\Sn}(t)$ is palindromic of degree $\rk(C_n)=n-1$, i.e.\ $\ch(Z_{C_n}^{\Sn}(t))[Y]=t^{n-1}\ch(Z_{C_n}^{\Sn}(t^{-1}))[Y]$.
Multiplication by $u^{n-1}$ followed by summation over~$n\ge 2$ yields the stated invariance of $\mathcal{Z}_C$.
\end{proof}

\section{Proof of Theorem~\ref{thm:main} by comparing \texorpdfstring{$Z$}{Z}-polynomials}\label{sec:z-comparison}

We prove Theorem~\ref{thm:main} by combining the orbit formula and palindromicity of the equivariant $Z$-polynomial (Definition~\ref{def:equivariant-kl-z-polynomials}) with the explicit formula for the $\Sn$-equivariant Kazhdan--Lusztig polynomial of $C_n$.
The proof does not use the equivariant characteristic polynomial of $T_n$ under $\Bn$, which we compute separately in Appendix~\ref{app:charpoly}.

For the thagomizer matroids, define the corresponding generating functions by
\begin{align*}
\Phi_T(t,u)
&:=\sum_{n\ge 0}\chB\!\left(P_{T_n}^{\Bn}(t)\right)\,u^{n+1}
=\sum_{n\ge 0}\mathsf{P}_n(X,Y;t)\,u^{n+1},\\
\mathcal{Z}_T(t,u)
&:=\sum_{n\ge 0}\chB\!\left(Z_{T_n}^{\Bn}(t)\right)\,u^{n+1}.
\end{align*}


We first express $\mathcal{Z}_T$ in terms of $\Phi_T$ and the generating functions for complete homogeneous symmetric functions.
A direct computation of the orbit sum, using the two orbit families in Lemma~\ref{lem:stabilizers}, gives the following identity.

\begin{prop}\label{prop:ZT-in-terms-of-PhiT}
We have
\begin{equation}\label{eq:ZT-PhiT}
\mathcal{Z}_T(t,u)=H_{X+Y}(tu)\,\Phi_T(t,u)+t u\,H_X(tu)\,H_X(u).
\end{equation}
\end{prop}
\begin{proof}
Write $Z_n^B(t):=\chB(Z_{T_n}^{\Bn}(t))\in \Sym[X,Y][t]$, so that $\mathcal{Z}_T(t,u)=\sum_{n\ge 0} Z_n^B(t)\,u^{n+1}$.
By Definition~\ref{def:equivariant-kl-z-polynomials}(iii),
\[
Z_n^B(t)=\sum_{[F]\in \cL(T_n)/\Bn} t^{\rk(F)}\,\chB\!\left(\Ind_{(\Bn)_F}^{\Bn} P_{T_n/F}^{(\Bn)_F}(t)\right).
\]
We compute the right side by summing the contributions of the two orbit families in Lemma~\ref{lem:stabilizers}.

We first consider the Type~I orbits, whose flats do not contain $e_*$.
For $k\in\{0,1,\dots,n\}$, take the representative $F_k^{\mathrm{I}}=\{a_1,\dots,a_k\}$.
Contracting $a_1,\dots,a_k$ identifies each vertex $i\le k$ with the vertex $A$.
For each $i\le k$, the remaining edge $b_i$ becomes an edge between $B$ and $A$, hence an element parallel to $e_*$.
The remaining $n-k$ spikes are unchanged.
Therefore $T_n/F_k^{\mathrm{I}}$ is obtained from $T_{n-k}$ by replacing the spine element by a parallel class of size $k+1$.
Deleting the $k$ additional parallel elements recovers~$T_{n-k}$.
The lattices of flats are naturally isomorphic under the action of $\B{n-k}$, so
\begin{equation}\label{eq:typeI-parallel-invariant}
P_{T_n/F_k^{\mathrm{I}}}^{\B{n-k}}(t)=P_{T_{n-k}}^{\B{n-k}}(t).
\end{equation}

By Lemma~\ref{lem:stabilizers}(a), the stabilizer is $(\Bn)_{F_k^{\mathrm{I}}}\cong \Sg{k}\times \B{n-k}$.
The $\Sg{k}$ factor permutes the $k$ parallel elements and fixes all other elements, hence acts trivially on $\cL(T_n/F_k^{\mathrm{I}})$.
Therefore the defining properties in Definition~\ref{def:equivariant-kl-z-polynomials} imply
\[
P_{T_n/F_k^{\mathrm{I}}}^{\Sg{k}\times \B{n-k}}(t)
=
\one_{\Sg{k}}\boxtimes P_{T_n/F_k^{\mathrm{I}}}^{\B{n-k}}(t).
\]
Combining this with~\eqref{eq:typeI-parallel-invariant}, we use Lemma~\ref{lem:ind-sk-bn-k} to compute the induced Frobenius characteristic as
\[
\chB\!\left(\Ind_{\Sg{k}\times \B{n-k}}^{\Bn}(\one_{\Sg{k}}\boxtimes P_{T_{n-k}}^{\B{n-k}}(t))\right)
=h_k[X+Y]\;\mathsf{P}_{n-k}(X,Y;t).
\]
Multiplying by the weight $t^{\rk(F_k^{\mathrm{I}})}=t^k$ shows that the Type~I contribution to $Z_n^B(t)$ is
\[
\sum_{k=0}^{n} t^k\,h_k[X+Y]\;\mathsf{P}_{n-k}(X,Y;t).
\]
Passing to generating functions, we compute
\begin{align*}
&\sum_{n\ge 0}\Bigl(\sum_{k=0}^{n} t^k\,h_k[X+Y]\;\mathsf{P}_{n-k}(X,Y;t)\Bigr)u^{n+1}\\
&= \sum_{k\ge 0} t^k\,h_k[X+Y]\sum_{m\ge 0}\mathsf{P}_m(X,Y;t)\,u^{m+k+1}
\qquad(m:=n-k)\\
&= \left(\sum_{k\ge 0} h_k[X+Y]\,(tu)^k\right)\left(\sum_{m\ge 0}\mathsf{P}_m(X,Y;t)\,u^{m+1}\right)\\
&= H_{X+Y}(tu)\,\Phi_T(t,u).
\end{align*}

For the Type~II orbits, whose flats contain $e_*$, fix $k\in\{0,1,\dots,n\}$ and take the representative $F_k^{\mathrm{II}}=\{e_*\}\cup\{a_1,b_1,\dots,a_k,b_k\}$.
Contracting $F_k^{\mathrm{II}}$ identifies $A$ with $B$ (because $e_*$ is contracted) and also contracts all edges in the first $k$ spikes.
For each remaining spike $i>k$, the edges $a_i$ and $b_i$ become a parallel pair between the identified vertex and $i$.
Hence $T_n/F_k^{\mathrm{II}}$ is a direct sum of $n-k$ matroids of rank one, so its lattice of flats is Boolean of rank $n-k$.
By Lemma~\ref{lem:boolean},
\[
P_{T_n/F_k^{\mathrm{II}}}^{\B{k}\times \B{n-k}}(t)=\one_{\B{k}}\boxtimes \one_{\B{n-k}},
\]
where we use Lemma~\ref{lem:stabilizers}(b) to identify $(\Bn)_{F_k^{\mathrm{II}}}\cong \B{k}\times \B{n-k}$.
Equation~\eqref{eq:frob-bn-ind} gives
\[
\chB\!\left(\Ind_{\B{k}\times \B{n-k}}^{\Bn}(\one_{\B{k}}\boxtimes \one_{\B{n-k}})\right)
=h_k[X]\,h_{n-k}[X].
\]
Multiplying by $t^{\rk(F_k^{\mathrm{II}})}=t^{k+1}$ gives the Type~II contribution
\(
\sum_{k=0}^{n} t^{k+1}\,h_k[X]\,h_{n-k}[X]
\)
to $Z_n^B(t)$.
Passing to generating functions, we compute
\begin{align*}
&\sum_{n\ge 0}\Bigl(\sum_{k=0}^{n} t^{k+1}\,h_k[X]\,h_{n-k}[X]\Bigr)u^{n+1}\\
&= t u\sum_{k\ge 0} t^k h_k[X]\,u^k \sum_{m\ge 0} h_m[X]\,u^m
\qquad(m:=n-k)\\
&= t u\left(\sum_{k\ge 0} h_k[X]\,(tu)^k\right)\left(\sum_{m\ge 0} h_m[X]\,u^m\right)\\
&= t u\,H_X(tu)\,H_X(u).
\end{align*}
Adding the contributions from the two orbit families yields~\eqref{eq:ZT-PhiT}.
\end{proof}

Let $R_n(X,Y;t)$ denote the right-hand side of \eqref{eq:main}, that is,
\[
R_n(X,Y;t):=h_n[X]+t\sum_{k=2}^{n} h_{n-k}[X]\,\mathsf{C}_k(Y;t),
\qquad(n\ge 0).
\]
Write its generating function as
\[
\Psi_T(t,u):=\sum_{n\ge 0} R_n(X,Y;t)\,u^{n+1}.
\]
Substituting $\Psi_T$ for $\Phi_T$ in~\eqref{eq:ZT-PhiT} gives a candidate for $\mathcal{Z}_T$. We verify its palindromicity by comparison with $\mathcal{Z}_C$.

\begin{lem}\label{lem:PhiT-PhiC}
We have
\begin{equation}\label{eq:PhiT-PhiC}
\Psi_T(t,u)=u\,H_X(u)\,\bigl(1+t u\,\Phi_C(t,u)\bigr).
\end{equation}
\end{lem}
\begin{proof}
Expanding the right side gives
\[
uH_X(u)=\sum_{n\ge 0} h_n[X]\,u^{n+1},
\qquad
uH_X(u)\cdot tu\,\Phi_C(t,u)
=t\sum_{a\ge 0}\sum_{k\ge 2} h_a[X]\,\mathsf{C}_k(Y;t)\,u^{a+k+1}.
\]
Thus the coefficient of $u^{n+1}$ in the second term is
\(
t\sum_{k=2}^{n} h_{n-k}[X]\,\mathsf{C}_k(Y;t).
\)
The coefficient of $u^{n+1}$ in the sum of the two terms is therefore $R_n(X,Y;t)$, proving~\eqref{eq:PhiT-PhiC}.
\end{proof}
We next combine Proposition~\ref{prop:ZC-in-terms-of-PhiC} and Lemma~\ref{lem:PhiT-PhiC} to eliminate $\Phi_C$ and express $\Psi_T$ in terms of $\mathcal{Z}_C$.

\begin{prop}\label{prop:ZT-ZC-relation}
In $\Sym[X,Y][t][[u]]$, we have
\begin{equation}\label{eq:ZT-ZC}
\Psi_T(t,u)
=\frac{u\,H_X(u)}{H_Y(tu)}\Bigl(tu\,\mathcal{Z}_C(t,u)+1+h_1[Y]\cdot tu\Bigr).
\end{equation}
\end{prop}
\begin{proof}
Lemma~\ref{lem:PhiT-PhiC} gives the first identity below. Multiplying~\eqref{eq:ZC-PhiC} by $tu$ and rearranging gives the second:
\[
\Psi_T(t,u)=u\,H_X(u)\bigl(1+tu\,\Phi_C(t,u)\bigr)
\quad\text{and}\quad
H_Y(tu)\bigl(1+tu\,\Phi_C(t,u)\bigr)=tu\,\mathcal{Z}_C(t,u)+1+h_1[Y]\cdot tu .
\]
Since $H_Y(tu)$ has constant term $1$, it is invertible in $\Sym[X,Y][t][[u]]$. Dividing by $H_Y(tu)$ gives~\eqref{eq:ZT-ZC}.
\end{proof}

We define
\[
\widetilde{\mathcal{Z}}_T(t,u):=H_{X+Y}(tu)\,\Psi_T(t,u)+tu\,H_X(tu)\,H_X(u).
\]
We use Proposition~\ref{prop:ZT-ZC-relation} to rewrite the defining expression for $\widetilde{\mathcal{Z}}_T$ and then apply Lemma~\ref{lem:pal-ZC} to obtain the following corollary.

\begin{cor}\label{cor:palindromicity-transfer}
	The series $\widetilde{\mathcal{Z}}_T$ is palindromic in the sense that
	\[
	\widetilde{\mathcal{Z}}_T(t,u)=\widetilde{\mathcal{Z}}_T(t^{-1},tu).
	\]
\end{cor}

\begin{proof}
By Proposition~\ref{prop:ZT-ZC-relation} and $H_{X+Y}(tu)=H_X(tu)\,H_Y(tu)$, we have
\[
H_{X+Y}(tu)\,\Psi_T(t,u)
=u\,H_X(tu)\,H_X(u)\Bigl(tu\,\mathcal{Z}_C(t,u)+1+h_1[Y]\cdot tu\Bigr).
\]
	Therefore
	\[
	\widetilde{\mathcal{Z}}_T(t,u)
	=u\,H_X(tu)\,H_X(u)\Bigl(tu\,\mathcal{Z}_C(t,u)+(1+t)+h_1[Y]\cdot tu\Bigr).
	\]
Under the change of variables $(t,u)\mapsto(t^{-1},tu)$, the bracketed expression transforms as follows:
\[
\begin{aligned}
tu\,\mathcal{Z}_C(t,u)+(1+t)+h_1[Y]\cdot tu
&\mapsto
u\,\mathcal{Z}_C(t^{-1},tu)+(1+t^{-1})+h_1[Y]\cdot u\\
&=t^{-1}\Bigl(tu\,\mathcal{Z}_C(t,u)+(1+t)+h_1[Y]\cdot tu\Bigr).
\end{aligned}
\]
The prefactor $uH_X(tu)H_X(u)$ acquires a factor of $t$, which cancels the factor $t^{-1}$ from the bracketed expression. Hence the product is invariant.
\end{proof}

\begin{proof}[Proof of Theorem~\ref{thm:main}]
We prove $\mathsf{P}_n(X,Y;t)=R_n(X,Y;t)$ by induction on $n$.
For $n=0$, the matroid $T_0$ has a Boolean lattice of flats, so $\mathsf{P}_0=1=h_0[X]=R_0$ by Lemma~\ref{lem:boolean}.
Fix $n\ge 1$ and assume $\mathsf{P}_m=R_m$ for all $m<n$.
Set
\[
\Delta(t,u):=\Phi_T(t,u)-\Psi_T(t,u)=\sum_{m\ge 0}\bigl(\mathsf{P}_m(X,Y;t)-R_m(X,Y;t)\bigr)\,u^{m+1}.
\]
By Proposition~\ref{prop:ZT-in-terms-of-PhiT},
\[
\mathcal{Z}_T(t,u)-\widetilde{\mathcal{Z}}_T(t,u)=H_{X+Y}(tu)\,\Delta(t,u).
\]
By the induction hypothesis, the coefficients of $u,u^2,\dots,u^n$ in $\Delta(t,u)$ vanish.
Since $H_{X+Y}(tu)$ has constant term $1$, it follows that the coefficient of $u^{n+1}$ in
$\mathcal{Z}_T-\widetilde{\mathcal{Z}}_T$ equals the coefficient of $u^{n+1}$ in $\Delta$, namely,
\[
f(t):=\mathsf{P}_n(X,Y;t)-R_n(X,Y;t).
\]

The argument of Lemma~\ref{lem:pal-ZC}, applied to $\mathcal{Z}_T$, together with Corollary~\ref{cor:palindromicity-transfer}, shows that
$\mathcal{Z}_T-\widetilde{\mathcal{Z}}_T$ is invariant under $(t,u)\mapsto(t^{-1},tu)$.
Taking coefficients of $u^{n+1}$ gives the ``palindromicity'' relation
\(
f(t)=t^{n+1}f(t^{-1}).
\)
In particular, the set of degrees in $t$ appearing in $f$ is symmetric about $\frac{n+1}{2}$. Thus any nonzero term of degree $k$
is paired with a nonzero term of degree $n+1-k$.
If $f\neq0$, let $d=\deg_t f$.
Then $d$ is paired with $n+1-d$, so $d\ge n+1-d$ and hence $d\ge\frac{n+1}{2}$.

On the other hand, the Kazhdan--Lusztig degree bound gives $\deg_t\mathsf{P}_n<(n+1)/2$.
For $R_n$, the same degree bound follows from its definition: $R_1=h_1[X]$ has degree zero, and Theorem~\ref{thm:cycle-input} gives $\deg_t\mathsf{C}_k=\lfloor k/2\rfloor-1$ for $k\ge2$.
Thus, for $n\ge2$,
\[
\deg_t R_n\le\max_{2\le k\le n}\bigl(1+\deg_t\mathsf{C}_k\bigr)
=\left\lfloor\frac n2\right\rfloor<\frac{n+1}{2}.
\]
In either case,
\[
\deg_t f\le\max\{\deg_t\mathsf{P}_n,\deg_t R_n\}<\frac{n+1}{2},
\]
contradicting the inequality above if $f\neq0$.
Therefore $f=0$, i.e.\ $\mathsf{P}_n(X,Y;t)=R_n(X,Y;t)$.
This completes the induction.
\end{proof}

\section{Equivariant inverse Kazhdan--Lusztig polynomials}\label{sec:inverse}

In this section we compute the equivariant inverse Kazhdan--Lusztig polynomial of $T_n$ under the action of $\Bn$ and thereby prove Theorem~\ref{thm:inverse}.
Using Proudfoot's equivariant inversion identity, we express $Q_{T_n}^{\Bn}(t)$ in terms of the equivariant Kazhdan--Lusztig polynomials of the matroids $C_k$ under the action of $\Sg{k}$ and rewrite the resulting coefficients using an alternating identity derived from Pieri's rule.
We then turn to $M(K_{2,n})$ and prove Corollary~\ref{cor:spine-deletion}.

For $n\ge 0$, set
\[
\mathsf{Q}_n(X,Y;t):=\chB\!\left(Q_{T_n}^{\Bn}(t)\right)\in \Sym[X,Y][t].
\]
We use Proudfoot's equivariant inverse Kazhdan--Lusztig polynomial $Q_\M^W(t)\in \grVRep(W)$~\cite{proudfoot2021incidence}.
The equivariant inversion identity implies that $Q_\M^W(t)$ also depends only on the $W$-action on $\cL(\M)$, and that the kernel of this action acts trivially on its coefficients.
Set
\[
\widehat Q_{\M}^{W}(t)=(-1)^{\rk\M}Q_{\M}^{W}(t).
\]
The corresponding incidence-algebra functions satisfy $\widehat Q*P=\delta$, where $*$ denotes convolution and $\delta$ its identity.
Evaluating this identity on $[\varnothing,E(\M)]$ gives the following orbit-sum formula.

\begin{lem}[\thmcite{proudfoot2021incidence}{Proposition~4.6}]\label{lem:proudfoot-inversion}
Let $W\curvearrowright\M$ be a loopless equivariant matroid of positive rank.
Then
\begin{equation}\label{eq:proudfoot-inversion}
\sum_{[F]\in \cL(\M)/W} (-1)^{\rk(\M|_F)}\Ind_{W_F}^{W}\!\left(Q_{\M|_F}^{W_F}(t)\otimes P_{\M/F}^{W_F}(t)\right)=0
\end{equation}
in $\grVRep(W)$.
\end{lem}
For rank zero, the sum equals $\one_W$ rather than zero.
Since $\rk T_n=n+1>0$ for all $n\ge0$, the zero-sum identity applies to every $T_n$.

We now specialize~\eqref{eq:proudfoot-inversion} to $W=\Bn$ and $\M=T_n$ and solve for $\mathsf{Q}_n$ in terms of $\mathsf{P}_i$.

\begin{lem}\label{lem:Q-from-P}
For all $n\ge 0$,
\begin{equation}\label{eq:Q-from-P}
\mathsf{Q}_n(X,Y;t)=\sum_{i=0}^{n} (-1)^i\,e_{n-i}[2X+Y]\;\mathsf{P}_i(X,Y;t).
\end{equation}
Here $2X$ denotes the sum of two copies of the alphabet $X$ in plethystic notation.
\end{lem}
\begin{proof}
Using the two families of orbits of flats from Lemma~\ref{lem:stabilizers}, we compute the corresponding contributions to~\eqref{eq:proudfoot-inversion} for $T_n$.

For flats not containing $e_*$, fix $k\in\{0,1,\dots,n\}$ and choose the representative $F_k^{\mathrm{I}}=\{a_1,\dots,a_k\}$.
Then $T_n|_{F_k^{\mathrm{I}}}$ is a Boolean matroid of rank $k$ with $\Sg{k}$ acting by permutations of the $k$ elements, and $(\Bn)_{F_k^{\mathrm{I}}}\cong \Sg{k}\times \B{n-k}$ by Lemma~\ref{lem:stabilizers}(a).
Write $B_k$ for the Boolean matroid of rank $k$.
The $\B{n-k}$ factor fixes every element of $T_n|_{F_k^{\mathrm{I}}}$, hence acts trivially on $Q_{T_n|_{F_k^{\mathrm{I}}}}^{\Sg{k}\times \B{n-k}}(t)$.
Moreover, $Q_{B_k}^{\Sg{k}}(t)=\mathrm{sgn}_{\Sg{k}}$ in degree $0$ for $k\ge1$~\cite[Theorem~3.4]{gao2022equivariant}.
The case $k=0$ follows from the normalization in rank zero, since $\mathrm{sgn}_{\Sg{0}}=\one_{\Sg{0}}$.
Since the $\Sg{k}$ factor acts trivially on $\cL(T_n/F_k^{\mathrm{I}})$, the defining properties in Definition~\ref{def:equivariant-kl-z-polynomials} imply
\(
P_{T_n/F_k^{\mathrm{I}}}^{\Sg{k}\times \B{n-k}}(t)=\one_{\Sg{k}}\boxtimes P_{T_n/F_k^{\mathrm{I}}}^{\B{n-k}}(t).
\)
Combining this identity with~\eqref{eq:typeI-parallel-invariant} gives $P_{T_n/F_k^{\mathrm{I}}}^{\Sg{k}\times \B{n-k}}(t)=\one_{\Sg{k}}\boxtimes P_{T_{n-k}}^{\B{n-k}}(t)$.
Therefore the Type I contribution to~\eqref{eq:proudfoot-inversion}, after applying $\chB$, is
\begin{equation}\label{eq:Q-from-P-typeI}
\sum_{k=0}^{n} (-1)^k\,e_k[X+Y]\;\mathsf{P}_{n-k}(X,Y;t),
\end{equation}
using Lemma~\ref{lem:ind-sk-bn-k} with $U=\mathrm{sgn}_{\Sg{k}}$ and $\ch(\mathrm{sgn}_{\Sg{k}})=e_k$.

For flats containing $e_*$, fix $k\in\{0,1,\dots,n\}$ and choose the representative
\(
F_k^{\mathrm{II}}=\{e_*\}\cup\{a_1,b_1,\dots,a_k,b_k\}.
\)
Then $T_n|_{F_k^{\mathrm{II}}}$ is canonically isomorphic to $T_k$ with its action of $\B{k}$, and $(\Bn)_{F_k^{\mathrm{II}}}\cong \B{k}\times \B{n-k}$ by Lemma~\ref{lem:stabilizers}(b).
The factor $\B{n-k}$ fixes every element of this restriction and hence acts trivially on $\cL(T_n|_{F_k^{\mathrm{II}}})$.
The contraction $T_n/F_k^{\mathrm{II}}$ has a Boolean lattice of flats of rank $n-k$, as in the proof of Proposition~\ref{prop:ZT-in-terms-of-PhiT}, so Lemma~\ref{lem:boolean} gives
\(
P_{T_n/F_k^{\mathrm{II}}}^{\B{k}\times \B{n-k}}(t)=\one_{\B{k}}\boxtimes \one_{\B{n-k}}.
\)
Inducing to $\Bn$ and applying $\chB$, we obtain the Type~II contribution
\begin{equation}\label{eq:Q-from-P-typeII}
\sum_{k=0}^{n} (-1)^{k+1}\,\mathsf{Q}_k(X,Y;t)\,h_{n-k}[X],
\end{equation}
where we use $\chB(\one_{\B{n-k}})=h_{n-k}[X]$ and the compatibility of $\chB$ with induction in~\eqref{eq:frob-bn-ind}.

Adding~\eqref{eq:Q-from-P-typeI} and~\eqref{eq:Q-from-P-typeII} and using~\eqref{eq:proudfoot-inversion} yields, for all $n\ge 0$,
\begin{equation}\label{eq:Q-from-P-convolution}
\sum_{k=0}^{n} (-1)^k\,\mathsf{Q}_k(X,Y;t)\,h_{n-k}[X]
=
\sum_{k=0}^{n} (-1)^k\,e_k[X+Y]\;\mathsf{P}_{n-k}(X,Y;t).
\end{equation}
Now multiply~\eqref{eq:Q-from-P-convolution} by $u^n$ and sum over $n\ge 0$.
Write
\[
E_A(u):=\sum_{n\ge 0} e_n[A]\,u^n.
\]
Then~\eqref{eq:Q-from-P-convolution} becomes
\begin{equation}\label{eq:Q-from-P-series}
\left(\sum_{n\ge 0} (-1)^n\mathsf{Q}_n(X,Y;t)\,u^n\right)H_X(u)
=
E_{X+Y}(-u)\left(\sum_{n\ge 0}\mathsf{P}_n(X,Y;t)\,u^n\right).
\end{equation}
Since $H_X(u)\,E_X(-u)=1$, we have $H_X(u)^{-1}=E_X(-u)$ and hence
\[
\sum_{n\ge 0} (-1)^n\mathsf{Q}_n(X,Y;t)\,u^n
=
E_{X+Y}(-u)\,E_X(-u)\left(\sum_{n\ge 0}\mathsf{P}_n(X,Y;t)\,u^n\right).
\]
Using $E_{X+Y}(-u)\,E_X(-u)=E_{2X+Y}(-u)$, we obtain
\[
\sum_{n\ge 0} (-1)^n\mathsf{Q}_n(X,Y;t)\,u^n
=
\left(\sum_{n\ge 0} (-1)^n e_n[2X+Y]\,u^n\right)\left(\sum_{n\ge 0}\mathsf{P}_n(X,Y;t)\,u^n\right).
\]
Comparing coefficients of $u^n$ gives~\eqref{eq:Q-from-P}.
\end{proof}

\begin{prop}\label{prop:inverse-cycle-input}
For all $n\ge 0$,
\begin{equation}\label{eq:main-inverse}
\mathsf{Q}_n(X,Y;t)=e_n[X+Y]+t\sum_{k=2}^{n} (-1)^k\,e_{n-k}[X+Y]\,\mathsf{C}_k(Y;t).
\end{equation}
\end{prop}
\begin{proof}
Substituting~\eqref{eq:main} into~\eqref{eq:Q-from-P}, we obtain
\[
\mathsf{Q}_n(X,Y;t)=\sum_{i=0}^{n} (-1)^i e_{n-i}[2X+Y]\,h_i[X]
\;+\;
t\sum_{i=2}^{n}(-1)^i e_{n-i}[2X+Y]\sum_{k=2}^{i} h_{i-k}[X]\,\mathsf{C}_k(Y;t).
\]
For the first sum, multiply by $u^n$ and sum over $n\ge 0$ to obtain $E_{2X+Y}(u)\,H_X(-u)$.
Since $H_X(-u)\,E_X(u)=1$ and $E_{2X+Y}(u)=E_X(u)^2 E_Y(u)$, we have
\[
E_{2X+Y}(u)\,H_X(-u)=E_X(u)E_Y(u)=E_{X+Y}(u).
\]
Thus
\[
\sum_{i=0}^{n} (-1)^i e_{n-i}[2X+Y]\,h_i[X]=e_n[X+Y].
\]
For the second sum, we write $i = k + m$ and interchange the order of summation to obtain
\begin{align*}
&t\sum_{i=2}^{n}(-1)^i e_{n-i}[2X+Y]\sum_{k=2}^{i} h_{i-k}[X]\,\mathsf{C}_k(Y;t)\\
&= t\sum_{k=2}^{n} \mathsf{C}_k(Y;t)\sum_{m=0}^{n-k}(-1)^{k+m} e_{n-k-m}[2X+Y]\,h_m[X] \\
&= t\sum_{k=2}^{n} (-1)^k\,e_{n-k}[X+Y]\;\mathsf{C}_k(Y;t),
\end{align*}
where in the last step we used the previous identity with $n$ replaced by $n-k$.
Combining the expressions for the two sums gives~\eqref{eq:main-inverse}.
\end{proof}

By the definition of $\chB$, Theorem~\ref{thm:inverse} is equivalent to the symmetric function identity
\begin{equation}\label{eq:inverse-induced}
\mathsf{Q}_n(X,Y;t)
=
\sum_{k=0}^{\left\lfloor\frac{n}{2}\right\rfloor}
\left(\sum_{i=2k}^{n} e_{n-i}[X]\,s_{(2^k,1^{i-2k})}[Y]\right)t^k.
\end{equation}

To convert the signed formula~\eqref{eq:main-inverse} into the positive decomposition in Theorem~\ref{thm:inverse}, we use the following alternating Pieri identity.

\begin{lem}[\thmcite{gao2025thagomizer}{Lemma~2.2}]\label{lem:pieri-alternating}
For $m\ge 0$ and $k\ge 1$, we have
\[
\sum_{j=0}^{m} (-1)^j\,e_{m-j}[Y]\,s_{(j+2,2^{k-1})}[Y]=s_{(2^k,1^m)}[Y].
\]
\end{lem}

\begin{proof}[Proof of Theorem~\ref{thm:inverse}]
We compute coefficients of $t^k$ using~\eqref{eq:main-inverse} and Theorem~\ref{thm:cycle-input}.
For $k=0$, the addition formula for elementary symmetric functions gives
\[
[t^0]\mathsf{Q}_n(X,Y;t)=e_n[X+Y]=\sum_{i=0}^{n} e_{n-i}[X]\,e_i[Y],
\]
which agrees with~\eqref{eq:inverse-induced} when $k=0$.

For $k>\lfloor n/2\rfloor$, the coefficients of $t^k$ on both sides of~\eqref{eq:inverse-induced} vanish by~\eqref{eq:main-inverse} and Theorem~\ref{thm:cycle-input}.
Now fix $1\le k\le\lfloor n/2\rfloor$.
We substitute the formula for $\mathsf{C}_m(Y;t)$ from Theorem~\ref{thm:cycle-input} into~\eqref{eq:main-inverse}. We then extract the coefficient of $t^k$ to obtain
\begin{align*}
[t^k]\mathsf{Q}_n(X,Y;t)
&=\sum_{m=2k}^{n} (-1)^m\,e_{n-m}[X+Y]\cdot s_{(m-2k+2,2^{k-1})}[Y]\\
&=\sum_{j=0}^{n-2k} (-1)^j\,e_{n-2k-j}[X+Y]\cdot s_{(j+2,2^{k-1})}[Y]
\qquad(j:=m-2k).
\end{align*}
In the preceding sum, we replace $e_{n-2k-j}[X+Y]$ by $\sum_{i=0}^{n-2k-j} e_i[X]\,e_{n-2k-j-i}[Y]$. We then sum over $j$ first to obtain
\begin{align*}
[t^k]\mathsf{Q}_n(X,Y;t)
&=\sum_{i=0}^{n-2k} e_i[X]\sum_{j=0}^{n-2k-i} (-1)^j\,e_{n-2k-i-j}[Y]\cdot s_{(j+2,2^{k-1})}[Y]\\
&=\sum_{i=0}^{n-2k} e_i[X]\cdot s_{(2^k,1^{n-2k-i})}[Y],
\end{align*}
by Lemma~\ref{lem:pieri-alternating} (with $m:=n-2k-i$).
Replacing the summation index $i$ by $n-i$ gives
\[
[t^k]\mathsf{Q}_n(X,Y;t)
=\sum_{i=2k}^{n} e_{n-i}[X]\cdot s_{(2^k,1^{i-2k})}[Y].
\]
This is the coefficient of $t^k$ in~\eqref{eq:inverse-induced}, so~\eqref{eq:inverse-induced} holds.
The statement about representations follows by applying $\chB^{-1}$ and the defining property $\chB(V_{\lambda,\mu})=s_\lambda[X]\,s_\mu[Y]$.
\end{proof}

We now prove the formulas for $M(K_{2,n})=T_n\setminus e_*$ stated in Corollary~\ref{cor:spine-deletion}.

\begin{proof}[Proof of Corollary~\ref{cor:spine-deletion}]
To prove the identities in~\eqref{eq:spine-deletion} for the Kazhdan--Lusztig and $Z$-polynomials, we apply the equivariant deletion formula of Ferroni, Matherne, and Vecchi~\cite[Theorem~1.1]{ferroni2025deletion} to $T_n$ and $e_*$.
The matroid is loopless, and $e_*$ is not a coloop since it belongs to the triangle $\{e_*,a_i,b_i\}$ for each $i$.
The stabilizer of $e_*$ in $\Bn$ is the entire group.

The correction terms in the deletion formula are indexed by the orbits of
\[
S_{e_*}(T_n)=\{F\in\cL(T_n):F\subsetneq E(T_n)\setminus\{e_*\},\ F\cup\{e_*\}\in\cL(T_n)\}.
\]
We have $S_{e_*}(T_n)=\{\varnothing\}$.
Indeed, by Lemma~\ref{lem:flats}(a), a flat avoiding $e_*$ contains at most one edge from each spike.
If such a flat $F$ is nonempty, choose, say, $a_i\in F$.
Then $b_i\notin F\cup\{e_*\}$, but the triangle $\{e_*,a_i,b_i\}$ implies that $b_i$ belongs to the closure of $F\cup\{e_*\}$.
Thus $F\cup\{e_*\}$ is not a flat.
On the other hand, both $\varnothing$ and $\{e_*\}$ are flats, so the empty flat is the unique possibility.

Contracting $e_*$ produces $n$ pairs of parallel edges, and hence
\[
T_n/e_*\cong\bigoplus_{i=1}^{n}U_{1,2}.
\]
Since its lattice of flats is Boolean of rank $n$, we apply Lemma~\ref{lem:boolean} to obtain
\[
P_{T_n/e_*}^{\Bn}(t)=\one_{\Bn}.
\]
Recall that, for an equivariant matroid $W\curvearrowright N$ of odd rank $2j+1$, the invariant $\tau^W(N)$ in the deletion formula is $[t^j]P_N^W(t)$. It is zero when the rank is even.
Consequently, $\tau^{\Bn}(T_n/e_*)=0$ for $n\ge2$. This follows from the definition when $n$ is even, and from $P_{T_n/e_*}^{\Bn}(t)=\one_{\Bn}$ when $n\ge3$ is odd.
The only possible correction term, corresponding to $F=\varnothing$, therefore vanishes in both deletion identities.
These identities reduce to
\[
P_{T_n}^{\Bn}(t)=P_{T_n\setminus e_*}^{\Bn}(t)-t\one_{\Bn},
\qquad
Z_{T_n}^{\Bn}(t)=Z_{T_n\setminus e_*}^{\Bn}(t),
\]
which give~\eqref{eq:spine-deletion}.
Using Theorem~\ref{thm:main} and the first identity, we obtain the explicit decomposition
\[
P_{T_n\setminus e_*}^{\Bn}(t)
=(1+t)V_{(n),\varnothing}
+\sum_{k=1}^{\lfloor n/2\rfloor}
\left(\sum_{i=2k}^{n}V_{(n-i),(i-2k+2,2^{k-1})}\right)t^k.
\]
The extra trivial representation in degree $1$ is distinct from all the other constituents in that degree, since their second partitions are nonempty.

We next prove the formula~\eqref{eq:inverse-spine-deletion} for the inverse Kazhdan--Lusztig polynomial.
For $n\ge1$, set
\[
\Delta_n(X,Y;t)
:=\chB\!\left(Q_{T_n\setminus e_*}^{\Bn}(t)\right)-\mathsf{Q}_n(X,Y;t).
\]
An edge set containing at most one edge from each spike is a flat of $T_n\setminus e_*$.
If a flat contains a complete spike, then $A$ and $B$ lie in the same connected component, so every spike it meets must be complete.
Conversely, every nonempty union of complete spikes is a flat.
The argument of Lemma~\ref{lem:stabilizers} therefore gives the nonempty orbit representatives $F_k^{\mathrm I}$ and
$F_k^{\mathrm{II}}\setminus\{e_*\}$, with $1\le k\le n$ in each family,
and the same ranks and stabilizers as for $T_n$.
For Type~I, the restrictions are unchanged, and the contractions replace
the spine of $T_{n-k}$ by $k$ parallel edges.
The $\Sg{k}$ factor permutes these edges and acts trivially on the lattice
of flats. Since $k\ge1$, this lattice is equivariantly isomorphic to that
of the corresponding contraction of $T_n$, so the contributions to
Lemma~\ref{lem:proudfoot-inversion} are unchanged.
For Type~II, the contractions still have Boolean lattices of flats, while
the restrictions change from $T_k$ to $T_k\setminus e_*$.

For $n\ge1$, Lemma~\ref{lem:proudfoot-inversion} therefore gives, for
$T_n$ and $T_n\setminus e_*$ respectively,
\begin{align*}
0&=\mathsf{P}_n
  +\sum_{k=1}^{n}(-1)^k e_k[X+Y]\mathsf{P}_{n-k}
  -h_n[X]
  +\sum_{k=1}^{n}(-1)^{k+1}\mathsf{Q}_k h_{n-k}[X],\\
0&=\chB\!\left(P_{T_n\setminus e_*}^{\Bn}(t)\right)
  +\sum_{k=1}^{n}(-1)^k e_k[X+Y]\mathsf{P}_{n-k}
  +\sum_{k=1}^{n}(-1)^{k+1}(\mathsf{Q}_k+\Delta_k)h_{n-k}[X],
\end{align*}
where the arguments $X,Y;t$ are omitted from $\mathsf{P}_k$, $\mathsf{Q}_k$,
and $\Delta_k$.
The term $-h_n[X]$ in the first equation is the contribution of the
rank-one flat $\{e_*\}$, which is absent after deletion.
We subtract the first equation from the second to obtain
\[
\sum_{k=1}^{n}(-1)^k\Delta_k(X,Y;t)\,h_{n-k}[X]
=h_n[X]+\chB\!\left(P_{T_n\setminus e_*}^{\Bn}(t)-P_{T_n}^{\Bn}(t)\right).
\]

By~\eqref{eq:spine-deletion}, which we have just proved, the last term
on the right is $t h_n[X]$ for $n\ge2$. For $n=1$, both Kazhdan--Lusztig polynomials are
trivial, so the displayed equation reduces to $-\Delta_1=h_1[X]$.
Writing $A(u):=\sum_{n\ge1}\Delta_n(X,Y;t)u^n$, we obtain
\[
A(-u)H_X(u)=(1+t)\bigl(H_X(u)-1\bigr)-t h_1[X]u.
\]
Since $H_X(-u)^{-1}=E_X(u)$, we replace $u$ by $-u$ and extract
coefficients to obtain, for $n\ge2$,
\[
\Delta_n(X,Y;t)
=-(1+t)e_n[X]+t h_1[X]e_{n-1}[X]
=-e_n[X]+t s_{(2,1^{n-2})}[X],
\]
where the last equality follows from Pieri's rule.
Applying $\chB^{-1}$ proves~\eqref{eq:inverse-spine-deletion}.
By Theorem~\ref{thm:inverse} and~\eqref{eq:inverse-spine-deletion},
the constant coefficient loses the constituent $V_{(1^n),\varnothing}$,
while the coefficient of $t$ gains $V_{(2,1^{n-2}),\varnothing}$.
The latter is not isomorphic to any of the other constituents in
degree $1$, all of which have nonempty second partitions.
Thus every coefficient remains multiplicity-free.

\end{proof}

For $n\ge3$, the circuits of $M(K_{2,n})$ are precisely the unions of two spikes.
Two elements in the same spike belong to exactly the same circuits, whereas elements in different spikes can be distinguished using a third spike.
Thus the matroid determines the spike pairs, and $\Aut(M(K_{2,n}))\cong\Bn$.
For $n=2$, the inherited group $\B{2}$ is a proper subgroup of
$\Aut(U_{3,4})\cong\mathfrak S_4$.

We take dimensions in~\eqref{eq:spine-deletion} to recover the known identity $P_{M(K_{2,n})}(t)=P_{T_n}(t)+t$ for $n\ge2$, which Gedeon, Proudfoot, and Young recorded in~\cite[Theorem~5.8]{gedeon2017survey}.
Similarly, we take dimensions in~\eqref{eq:inverse-spine-deletion} to obtain
\[
Q_{M(K_{2,n})}(t)=Q_{T_n}(t)-1+(n-1)t
\qquad(n\ge2).
\]

\section{Strong induced log-concavity}\label{sec:ilc}

In this section, we prove Theorem~\ref{thm:strong-induced-logconcavity}
by reducing the required inequalities to Schur positivity for two families
of truncated Schur sums. We prove the inequalities for the first family
using Speyer's criterion. For the second family, we express the Schur
coefficients on both sides as sums of cardinalities of integer intervals
and compare these cardinalities term by term.

\subsection{Reduction to truncated Schur sums}

We begin by recalling the relevant definitions from~\cite{gao2026induced}.

\begin{defn}[\thmcite{gao2026induced}{\S1}]\label{def:ilc-bn}
Let $W$ be a finite group and let $(V_k)_{k\ge0}$ be a sequence in $\VRep(W)$.
Let $H$ be a subgroup of a finite group $G$, and let
$\phi\colon H\to W\times W$ be a group homomorphism.
We regard $V_i\boxtimes V_j$ as a virtual representation of $H$ by pullback along $\phi$.
\begin{itemize}
\item
The sequence $(V_k)_{k\ge0}$ is \emph{induced log-concave} with respect to
$(G,H,\phi)$ if, for every $i\ge1$,
\[
\Ind_H^G(V_i\boxtimes V_i)
-\Ind_H^G(V_{i-1}\boxtimes V_{i+1})
\]
is an honest representation of $G$.
\item
The sequence $(V_k)_{k\ge0}$ is \emph{strongly induced log-concave} with respect to
$(G,H,\phi)$ if, for all $1\le i\le j$,
\[
\Ind_H^G(V_i\boxtimes V_j)
-\Ind_H^G(V_{i-1}\boxtimes V_{j+1})
\]
is an honest representation of $G$.
\end{itemize}
\end{defn}

We take $W=\Bn$, $G=\B{2n}$, and the standard block subgroup
$H=\Bn\times\Bn$, with $\phi=\mathrm{id}$.
By~\eqref{eq:frob-bn-ind}, induction from $H$ to $G$ corresponds to
multiplication in $\Sym[X,Y]$.
Since $\chB(V_{\lambda,\mu})=s_\lambda[X]s_\mu[Y]$, a virtual representation
is honest if and only if its Frobenius characteristic is Schur-positive.

Write $f\succeq_s g$ if $f-g$ is a nonnegative linear combination of
Schur functions $s_\lambda$ in one alphabet, or of products
$s_\lambda[X]s_\mu[Y]$ in two alphabets.
Throughout this section, we take empty sums to be zero and omit trailing zero
parts from partitions.
Set $\mathsf{P}_{n,k}=[t^k]\mathsf{P}_n(X,Y;t)$ and
$\mathsf{Q}_{n,k}=[t^k]\mathsf{Q}_n(X,Y;t)$.
Theorems~\ref{thm:main} and~\ref{thm:inverse} give
$\mathsf{P}_{n,0}=h_n[X]$ and
\begin{align*}
\mathsf{P}_{n,k}
&=\sum_{p=2k}^{n}h_{n-p}[X]s_{(p-2k+2,2^{k-1})}[Y]
&&\text{for }k\ge1,\\
\mathsf{Q}_{n,k}
&=\sum_{p=2k}^{n}e_{n-p}[X]s_{(2^k,1^{p-2k})}[Y]
&&\text{for }k\ge0.
\end{align*}
It remains to show that, for every $n\ge0$ and $1\le i\le j$,
\[
\mathsf{P}_{n,i}\mathsf{P}_{n,j}\succeq_s
\mathsf{P}_{n,i-1}\mathsf{P}_{n,j+1},
\qquad
\mathsf{Q}_{n,i}\mathsf{Q}_{n,j}\succeq_s
\mathsf{Q}_{n,i-1}\mathsf{Q}_{n,j+1}.
\]

To describe the coefficients of the Schur functions in $X$, we introduce
two families of truncated Schur sums in a single alphabet. We suppress
this alphabet from the notation.
By Pieri's rule, the degrees of the two Schur factors in $Y$ have a common
upper bound $N$ once we fix the Schur term in $X$.
For each integer $N\ge0$, define
\begin{equation}\label{eq:ilc-A-definition}
A_{N,0}=1,\qquad
A_{N,k}=\sum_{p=2k}^{N}s_{(p-2k+2,2^{k-1})}
\quad\text{for }k\ge1.
\end{equation}
We also define
\begin{equation}\label{eq:ilc-B-definition}
B_{N,k}=\sum_{p=2k}^{N}s_{(p-k,k)}
\qquad\text{for }k\ge0.
\end{equation}
The convention $A_{N,0}=1$ records the fact that
$\mathsf{P}_{n,0}=h_n[X]$ has $Y$-degree zero.

By Pieri's rule,
\[
h_{n-p}h_{n-q}
=\sum_{r=0}^{\min\{n-p,n-q\}}s_{(2n-p-q-r,r)}.
\]
For $0\le d\le2n$ and $0\le r\le\lfloor d/2\rfloor$, the coefficient
of $s_{(d-r,r)}[X]$ in
$\mathsf{P}_{n,i}\mathsf{P}_{n,j}-\mathsf{P}_{n,i-1}\mathsf{P}_{n,j+1}$
is the homogeneous component of degree $2n-d$ of the following symmetric
function in $Y$:
\[
A_{n-r,i}[Y]A_{n-r,j}[Y]-A_{n-r,i-1}[Y]A_{n-r,j+1}[Y].
\]
Indeed, we fix the $X$-shape $(d-r,r)$ in Pieri's rule and obtain the
constraints $p+q=2n-d$ and $p,q\le n-r$.
We therefore take $N=n-r$ in both truncated Schur sums.
The Schur factors indexed by $p$ and $q$ have $Y$-degrees $p$ and $q$, respectively.
Thus the condition $p+q=2n-d$ selects the homogeneous component of degree $2n-d$.
The Schur expansion in $X$ involves only partitions with at most two parts.
Our convention for $A_{N,0}$ ensures that this coefficient formula also
holds when $i=1$.

Define the standard involution $\omega$ by
$\omega(s_\lambda)=s_{\lambda'}$, where $\lambda'$ is the conjugate partition of $\lambda$.
We apply $\omega$ in each alphabet to $\mathsf{Q}_{n,k}$ and obtain
\[
\sum_{p=2k}^{n}h_{n-p}[X]s_{(p-k,k)}[Y].
\]
We then apply the same Pieri expansion to products of these sums and obtain
the preceding coefficient formula with $B$ in place of $A$.
To take the homogeneous component of degree $2n-d$ in $Y$, we retain only
the Schur terms of that degree. Their coefficients do not change, so this
operation preserves Schur positivity. Since $\omega$ also preserves Schur
positivity, it suffices to prove, for every $N\ge0$ and $1\le i\le j$,
\begin{equation}\label{eq:ilc-truncated-goal}
A_{N,i}A_{N,j}\succeq_s A_{N,i-1}A_{N,j+1},
\qquad
B_{N,i}B_{N,j}\succeq_s B_{N,i-1}B_{N,j+1}.
\end{equation}

\subsection{The family \texorpdfstring{$A$}{A}}

For the family $A$ in~\eqref{eq:ilc-A-definition}, we use Speyer's
criterion~\cite[Corollary~1]{Speyer2026}.
Take partitions $\lambda,\mu,\widetilde\lambda,\widetilde\mu$,
pad them with zeros to a common length, and suppose that
$\lambda+\mu=\widetilde\lambda+\widetilde\mu$.
If
\begin{align*}
\min\{\lambda_r,\mu_r\}&\le\widetilde\lambda_r\le\max\{\lambda_r,\mu_r\}
&&\text{for every }r,\\
\min\{\lambda_r-\lambda_s,\mu_r-\mu_s\}
&\le\widetilde\lambda_r-\widetilde\lambda_s
\le\max\{\lambda_r-\lambda_s,\mu_r-\mu_s\}
&&\text{for every }r<s,
\end{align*}
then $s_{\widetilde\lambda}s_{\widetilde\mu}\succeq_s s_\lambda s_\mu$.
When we pad the partitions with additional zeros, each new difference
bound either reduces to a coordinate bound above or is trivial.
We can therefore apply the criterion in arbitrarily many variables to obtain
the stated inequality in the ring of symmetric functions.

\begin{prop}\label{prop:ilc-A}
For integers $N\ge0$ and $1\le i\le j$,
\[
A_{N,i}A_{N,j}\succeq_s A_{N,i-1}A_{N,j+1}.
\]
\end{prop}

\begin{proof}
If $N<2j+2$, then $A_{N,j+1}=0$ and the result is immediate.
Assume $N\ge2j+2$.
For $i=1$, we have $A_{N,0}=1$ and $A_{N,1}\succeq_s h_2$.
For $2\le s\le N-2j$, it follows from Pieri's rule that
$h_2s_{(s,2^{j-1})}\succeq_s s_{(s,2^j)}$, since the two boxes in the
new bottom row form a horizontal strip. Hence
$A_{N,1}A_{N,j}\succeq_s h_2A_{N-2,j}\succeq_s A_{N,j+1}$,
which proves the case $i=1$.

For $i\ge2$, let $R=N-2i+2$ and $S=N-2j+2$, so that
$R\ge S\ge4$. We prove
\[
A_{N,i}A_{N,j}\succeq_s A_{N+2,i}A_{N-2,j}
\succeq_s A_{N,i-1}A_{N,j+1}.
\]
The pairs of first-row lengths in $A_{N+2,i}A_{N-2,j}$ and
$A_{N,i-1}A_{N,j+1}$ range over the same rectangle
$[2,R+2]\times[2,S-2]$. We can therefore
change the upper summation bounds first and then compare the partition
shapes term by term.
To prove the first inequality, define $(r',s')$ for each integer pair
$(r,s)\in[2,R+2]\times[2,S-2]$ by
\[
(r',s')=
\begin{cases}
(r,s),&r\le R,\\
(R-S+2+s,S-2+\varepsilon),&r=R+\varepsilon,\quad\varepsilon=1,2.
\end{cases}
\]
When $r>R$, the second coordinate is $S-1$ or $S$, while the first
runs through $R-S+4,\ldots,R$. These pairs are distinct and do not
belong to $[2,R]\times[2,S-2]$, so the map $(r,s)\mapsto(r',s')$
is an injection into $[2,R]\times[2,S]$.

For $r>R$, we have $r+s=r'+s'$ and $s\le r'\le r$.
To apply Speyer's criterion, take
\[
\begin{aligned}
\lambda&=(r,2^{i-1}), & \mu&=(s,2^{j-1}),\\
\widetilde\lambda&=(r',2^{i-1}), & \widetilde\mu&=(s',2^{j-1}).
\end{aligned}
\]
The componentwise sums agree. The coordinate bounds follow from
$s\le r'\le r$ and the equality $\widetilde\lambda_a=\lambda_a$ for $a\ge2$.
For $2\le b\le i$, we have
$\widetilde\lambda_1-\widetilde\lambda_b=r'-2\in[s-2,r-2]$.
For $i<b\le j$ and $b>j$, the difference is $r'$, which lies in
$[s-2,r]$ and $[s,r]$, respectively.
For $2\le a<b$, we have
$\widetilde\lambda_a-\widetilde\lambda_b=\lambda_a-\lambda_b$,
so the remaining difference bounds also hold.
It follows from Speyer's criterion that
\[
s_{(r',2^{i-1})}s_{(s',2^{j-1})}
\succeq_s s_{(r,2^{i-1})}s_{(s,2^{j-1})}.
\]
For $r\le R$, the same inequality holds with equality because
$(r',s')=(r,s)$.

Since the pairs $(r',s')$ are distinct and lie in $[2,R]\times[2,S]$,
and every product of Schur functions is Schur-positive, we obtain
\[
\begin{aligned}
A_{N,i}A_{N,j}
&\succeq_s
\sum_{r=2}^{R+2}\sum_{s=2}^{S-2}
s_{(r',2^{i-1})}s_{(s',2^{j-1})}\\
&\succeq_s
\sum_{r=2}^{R+2}\sum_{s=2}^{S-2}
s_{(r,2^{i-1})}s_{(s,2^{j-1})}
=A_{N+2,i}A_{N-2,j}.
\end{aligned}
\]

For the second inequality, we show that, for every
$(r,s)\in[2,R+2]\times[2,S-2]$,
\[
s_{(r,2^{i-1})}s_{(s,2^{j-1})}
\succeq_s s_{(r,2^{i-2})}s_{(s,2^j)}.
\]
We apply Speyer's criterion to the conjugate partitions
\[
\begin{aligned}
\lambda&=(i-1,i-1,1^{r-2}), & \mu&=(j+1,j+1,1^{s-2}),\\
\widetilde\lambda&=(i,i,1^{r-2}), & \widetilde\mu&=(j,j,1^{s-2}).
\end{aligned}
\]
The componentwise sums agree, and the first two coordinate bounds follow
from $i-1\le i\le j+1$.
For each coordinate beyond the first two, write the corresponding entries
of $\lambda$ and $\mu$ as $\varepsilon,\eta\in\{0,1\}$.
The lower difference bound $i-1-\varepsilon\le i-\varepsilon$ is immediate,
and the upper bound follows from
\[
j+1-\eta-(i-\varepsilon)
=j-i+1+\varepsilon-\eta\ge j-i\ge0.
\]
These bounds also hold when $i=j$.
All remaining coordinates are unchanged, and the first two coordinates
of each partition are equal. Thus all the hypotheses of Speyer's criterion hold.
We apply $\omega$ to transfer the resulting inequality back to the original
partitions. We then sum over the rectangle to prove the second inequality.
\end{proof}

\subsection{The family \texorpdfstring{$B$}{B}}

For the family $B$ in~\eqref{eq:ilc-B-definition}, we keep the sum $\ell$
of the two indices fixed and show that each Schur coefficient of
$B_{N,h}B_{N,\ell-h}$ is nondecreasing as $h$ increases through the integers
in $[0,\ell/2]$. Only partitions with at most four parts occur, since each
Schur factor is indexed by a partition with at most two parts.
The next lemma expresses these coefficients as sums of interval
cardinalities. We compare the summands as $h$ increases.

Write $[a,b]_{\mathbb Z}=[a,b]\cap\mathbb Z$. We take this interval to be
empty when $a>b$. Let $r_+=\max\{r,0\}$.
Let $N,\ell\ge0$ be integers, and let $\nu=(u,v,w,z)$ be a partition
of $m$ with at most four parts, padded with zeros if necessary.
Fix an integer $h$ with $0\le h\le\ell/2$.
Set
\[
L=\max\{0,\ell-v-z\},\qquad
U_h=\min\{u-v,w-z,h-z,\ell-w-z\},
\]
and, for each integer $t$, define
\[
I_h(t)=
\bigl[v-\ell+z+h+\max\{w,h\}+t,\,
      u+h-(\ell-h-w)_+-t\bigr]_{\mathbb Z}.
\]

For each integer $t$ with $L\le t\le U_h$, the interval
$I_h(t)$ records the possible degrees $a+h$ of the Schur term $s_{(a,h)}$
before we impose the degree bound $N$.
Only pairs of Schur terms whose degrees sum to $m$ contribute to the
coefficient of $s_\nu$. Thus we intersect $I_h(t)$ with $[m-N,N]_{\mathbb Z}$
to ensure that neither degree exceeds $N$.

\begin{lem}\label{lem:ilc-B-coeff}
If $2\ell\le m\le2N$, then
\begin{equation}\label{eq:ilc-B-count}
[s_\nu]B_{N,h}B_{N,\ell-h}
=\sum_{t=L}^{U_h}\bigl|I_h(t)\cap[m-N,N]_{\mathbb Z}\bigr|
\qquad(0\le h\le\ell/2).
\end{equation}
\end{lem}

\begin{proof}
Let $c^\nu_{\lambda,\mu}$ denote the Littlewood--Richardson coefficient.
Fix $h$ and let $k=\ell-h$. We expand the two factors to obtain
\[
[s_\nu]B_{N,h}B_{N,k}
=\sum_{\substack{a+b=m-\ell\\
                  a\ge h,\ b\ge k\\
                  a+h\le N,\ b+k\le N}}
c^\nu_{(a,h),(b,k)}.
\]
For each $a$, we express $c^\nu_{(a,h),(b,k)}$ as the number of integers $t$
satisfying certain bounds. Since $b=m-\ell-a$, the sum above counts pairs
$(a,t)$.

For partitions $(a,h)$ and $(b,k)$ with $a+b+\ell=m$, we use the
skew Schur expansion
\[
s_{\nu/(a,h)}(x,y)
=\sum_{r\ge s\ge0}c^\nu_{(a,h),(r,s)}s_{(r,s)}(x,y).
\]
The Schur polynomials on the right are linearly independent, so the
coefficient of $s_{(b,k)}(x,y)$ is $c^\nu_{(a,h),(b,k)}$.
We take the skew Schur polynomial to be zero when $(a,h)\nsubseteq\nu$.

A semistandard tableau of shape $\nu/(a,h)$ with entries in $\{1,2\}$
corresponds to a chain $(a,h)\subseteq\kappa\subseteq\nu$ in which
both successive differences are horizontal strips.
The partition $\kappa$ has at most three parts, and the interlacing conditions are
\[
\max\{v,a\}\le\kappa_1\le u,\qquad
\max\{w,h\}\le\kappa_2\le\min\{v,a\},\qquad
z\le\kappa_3\le\min\{w,h\}.
\]
These three intervals are nonempty precisely when
$\max\{h,w\}\le a\le u$ and $z\le h\le v$.
Under these conditions, we may choose $\kappa_1,\kappa_2,\kappa_3$ independently
from their respective intervals. The bounds ensure that
$\kappa_1\ge\kappa_2\ge\kappa_3$.
If any of the intervals is empty, the skew Schur polynomial is zero.
We therefore assume that all three intervals are nonempty.

Define
\begin{align*}
d_1&=u-\max\{v,a\},&
d_2&=\min\{v,a\}-\max\{w,h\},&
d_3&=\min\{w,h\}-z,
\end{align*}
and $D=z+(v-a)_++(w-h)_+$.
Among the tableau weights $x^{|\kappa|-a-h}y^{m-|\kappa|}$,
the minimum exponent of $x$ and the minimum exponent of $y$ are both $D$.
We factor $(xy)^D$ out of each tableau weight.
The independent choices in the three intervals then contribute
$h_{d_1}(x,y)$, $h_{d_2}(x,y)$, and $h_{d_3}(x,y)$, respectively. Thus
\[
s_{\nu/(a,h)}(x,y)
=(xy)^D h_{d_1}(x,y)h_{d_2}(x,y)h_{d_3}(x,y).
\]
The degrees satisfy
\[
2D+d_1+d_2+d_3=m-a-h=b+k.
\]

To extract the coefficient of $s_{(b,k)}(x,y)$ from this product, we use
Pieri's rule in two variables:
\[
h_r(x,y)h_s(x,y)
=\sum_{t=0}^{\min\{r,s\}}(xy)^t h_{r+s-2t}(x,y)
\qquad(r,s\ge0).
\]
We first apply this identity to $h_{d_1}h_{d_3}$, with summation index $t$.
We then apply it to $h_{d_1+d_3-2t}h_{d_2}$, with summation index $r$,
and obtain the term
\[
(xy)^{D+t+r}h_{d_1+d_2+d_3-2t-2r}(x,y).
\]

Since $s_{(b,k)}(x,y)=(xy)^k h_{b-k}(x,y)$, this term contributes $1$ to
$c^\nu_{(a,h),(b,k)}$ precisely when $D+t+r=k$.
With $K=k-D$, we therefore have $r=K-t$.
By the degree identity above, $d_1+d_2+d_3-2K=b-k$.
The two applications of Pieri's rule require
\[
0\le t\le\min\{d_1,d_3\},\qquad
0\le K-t\le\min\{d_1+d_3-2t,d_2\}.
\]
We combine these inequalities to express $c^\nu_{(a,h),(b,k)}$ as the
number of integers $t$ satisfying
\[
\max\{0,K-d_2\}\le t\le
\min\{d_1,d_3,K,d_1+d_3-K\}.
\]
Since $K-d_2=\ell-v-z$, the lower bound on $t$ is independent of $a$ and $h$.

We now fix $t$ and determine the possible values of $a$.
We use the identity $d_1+d_3-K=u+w-a-k$.
The bounds $t\le d_1,d_3$ are equivalent to
$t\le\min\{u-v,w-z,h-z\}$ and $a\le u-t$.
Similarly, $t\le K$ is equivalent to
\[
t\le\min\{k-z,\ell-w-z\},\qquad
a\ge v-\ell+z+\max\{w,h\}+t.
\]
Since $t\le h-z\le k-z$, we may omit the bound $t\le k-z$.
The bound $t\le d_1+d_3-K$ is equivalent to $a\le u+w-k-t$.
We can now write the Pieri bounds in terms of $a$ and $t$ as
\begin{equation}\label{eq:ilc-B-region}
L\le t\le U_h,\qquad
v-\ell+z+\max\{w,h\}+t\le a\le u-(k-w)_+-t.
\end{equation}

We next verify the original partition and interlacing conditions for
every integer pair $(a,t)$ satisfying~\eqref{eq:ilc-B-region}.
The lower bounds on $a$ and $t$ imply $a\ge\max\{w,h\}$.
Also, $h\ge z$ because $0\le t\le h-z$.
Since $t\ge\ell-v-z$ and
$t\le\min\{k-z,\ell-w-z\}=k-z-(w-h)_+$, we also have $h\le v$.
By the upper bound on $a$, we have $a\le u$ and $a\le u+w-k-t$.
Since $t\ge\ell-v-z$, we have
\[
(m-\ell-k)-(u+w-k-t)=v+z-\ell+t\ge0.
\]
Therefore
\[
a\le u+w-k-t\le m-\ell-k,\qquad b=m-\ell-a\ge k.
\]
Thus both $(a,h)$ and $(b,k)$ are partitions, and the three intervals for
$\kappa_1,\kappa_2,\kappa_3$ are nonempty.
We can therefore reverse the preceding
transformations of the Pieri bounds.
For each fixed $a$, the number of integers $t$
satisfying~\eqref{eq:ilc-B-region} is $c^\nu_{(a,h),(b,k)}$.

For fixed $t$, let $q=a+h$ be the degree of $s_{(a,h)}$. The other Schur
factor $s_{(b,k)}$ then has degree $b+k=m-q$.
The bounds on $a$ in~\eqref{eq:ilc-B-region} become $q\in I_h(t)$,
while the degree bounds $q\le N$ and $m-q\le N$ are equivalent to
$q\in[m-N,N]_{\mathbb Z}$.
Each integer $q$ in the intersection determines a unique pair $(a,b)$.
Thus the sum in~\eqref{eq:ilc-B-count} equals
$[s_\nu]B_{N,h}B_{N,k}$.
\end{proof}

\begin{prop}\label{prop:ilc-B}
For integers $N\ge0$ and $1\le i\le j$,
\[
B_{N,i}B_{N,j}\succeq_s B_{N,i-1}B_{N,j+1}.
\]
\end{prop}

\begin{proof}
Let $\ell=i+j$ and fix $\nu=(u,v,w,z)$ of size $m$.
The coefficients of $s_\nu$ in both products vanish unless
$2\ell\le m\le2N$. Assume these inequalities hold.
By Lemma~\ref{lem:ilc-B-coeff}, it suffices to show that the sum
in~\eqref{eq:ilc-B-count} is nondecreasing as $h$ increases through
the integers in $[0,\ell/2]$.

Fix an integer $t$ and consider the integers $h$ for which
$0\le h\le\ell/2$ and $L\le t\le U_h$.
Write $I_h(t)=[\alpha_h(t),\beta_h(t)]_{\mathbb Z}$, where
$\beta_h(t)=u+h-(\ell-h-w)_+-t$.
Since $2h\le\ell$, we have
$\max\{w,h\}+(\ell-h-w)_+=\ell-\min\{h,w,\ell-w\}$. Hence
\[
\beta_h(t)-\alpha_h(t)+1
=(u-v-t)+\bigl(\min\{h,w,\ell-w\}-z-t\bigr)+1\ge1.
\]
Both parenthesized terms are nonnegative by $t\le U_h$.
Thus $I_h(t)$ is nonempty, and $|I_h(t)|=\beta_h(t)-\alpha_h(t)+1$.
Moreover,
\[
\alpha_h(t)+\beta_h(t)
=m+2h-\ell+(h-w)_+-(\ell-h-w)_+\le m,
\]
since $2h\le\ell$ and $(h-w)_+\le(\ell-h-w)_+$.

Therefore $N-\alpha_h(t)+1\ge\beta_h(t)-m+N+1$.
The intersection has cardinality
\[
\begin{aligned}
\bigl|I_h(t)\cap[m-N,N]_{\mathbb Z}\bigr|
&=\max\bigl\{0,\min\{\beta_h(t),N\}
                 -\max\{\alpha_h(t),m-N\}+1\bigr\}\\
&=\max\Bigl\{0,\min\bigl\{|I_h(t)|,\ \beta_h(t)-m+N+1,
                         \ 2N-m+1\bigr\}\Bigr\}.
\end{aligned}
\]
Both $|I_h(t)|$ and $\beta_h(t)-m+N+1$ are nondecreasing in $h$,
while $2N-m+1$ is constant. Hence
$\bigl|I_h(t)\cap[m-N,N]_{\mathbb Z}\bigr|$ is also nondecreasing in $h$.

Since $L$ is independent of $h$ and $U_h$ is nondecreasing, the summation
range in~\eqref{eq:ilc-B-count} can only expand as $h$ increases.
Each existing summand is nondecreasing, and any new summands are nonnegative.
The coefficient of $s_\nu$ for $h=i$ is therefore at least that for $h=i-1$.
Since $\ell=i+j$, this proves the proposition.
\end{proof}

\begin{proof}[Proof of Theorem~\ref{thm:strong-induced-logconcavity}]
The theorem follows from Propositions~\ref{prop:ilc-A} and~\ref{prop:ilc-B}
by the reduction to~\eqref{eq:ilc-truncated-goal}.
\end{proof}

\section*{Acknowledgements}
All authors contributed equally and are listed in alphabetical order by surname.
Matthew H.~Y.~Xie and Michael X.~X.~Zhong were partially supported by the National Natural Science Foundation of China (Grant No.~12271403).
Philip B.~Zhang was supported by the National Natural Science Foundation of China (No.~12171362) and Natural Science Foundation of Tianjin Municipality (No.~\mbox{25JCYBJC00430}).


\begin{thebibliography}{10}

\bibitem{athanasiadis2020some}
C.~A. Athanasiadis, Some applications of {R}ees products of posets to equivariant gamma-positivity, \emph{Algebr. Comb.}, \textbf{3} (2020), 281--300.


\bibitem{braden2020singular}
T.~Braden, J.~Huh, J.~P. Matherne, N.~Proudfoot, and B.~Wang, Singular Hodge theory for combinatorial geometries, \emph{J. Amer. Math. Soc.}, to appear.

\bibitem{braden2020deletion}
T.~Braden and A.~Vysogorets, Kazhdan--{L}usztig polynomials of matroids under deletion, \emph{Electron. J. Combin.}, \textbf{27} (2020), \#P1.17.

\bibitem{elias2016kazhdan}
B.~Elias, N.~Proudfoot, and M.~Wakefield, The {K}azhdan--{L}usztig polynomial of a matroid, \emph{Adv. Math.}, \textbf{299} (2016), 36--70.

\bibitem{ferroni2025deletion}
L.~Ferroni, J.~P. Matherne, and L.~Vecchi, Deletion formulas for equivariant {K}azhdan--{L}usztig polynomials of matroids, \emph{SIAM J. Discrete Math.}, \textbf{39} (2025), 848--862.

\bibitem{gao2026induced}
A.~L.~L. Gao, E.~Y.~H. Li, M.~H.~Y. Xie, A.~L.~B. Yang, and Z.-X. Zhang, Induced log-concavity of equivariant matroid invariants, \emph{J. Comb. Algebra}, (2026), published online first.

\bibitem{gao2025thagomizer}
A.~L.~L. Gao, Y.~Li, and M.~H.~Y. Xie, Equivariant inverse Kazhdan--Lusztig polynomials of thagomizer matroids, \emph{Discrete Appl. Math.}, \textbf{395} (2026), 525--542.

\bibitem{gao2021inverse}
A.~L.~L. Gao and M.~H.~Y. Xie, The inverse {K}azhdan--{L}usztig polynomial of a matroid, \emph{J. Combin. Theory Ser. B}, \textbf{151} (2021), 375--392.

\bibitem{gao2022equivariant}
A.~L.~L. Gao, M.~H.~Y. Xie, and A.~L.~B. Yang, The equivariant inverse {K}azhdan--{L}usztig polynomials of uniform matroids, \emph{SIAM J. Discrete Math.}, \textbf{36} (2022), 2553--2569.

\bibitem{gedeon2017equivariant}
K.~Gedeon, N.~Proudfoot, and B.~Young, The equivariant {K}azhdan--{L}usztig polynomial of a matroid, \emph{J. Combin. Theory Ser. A}, \textbf{150} (2017), 267--294.

\bibitem{gedeon2017survey}
K.~Gedeon, N.~Proudfoot, and B.~Young, Kazhdan--Lusztig polynomials of matroids: a survey of results and conjectures, \emph{S\'em. Lothar. Combin.}, \textbf{78B} (2017), Article~80, 12~pp.

\bibitem{gedeon2017thagomizer}
K.~R. Gedeon, Kazhdan--{L}usztig polynomials of thagomizer matroids, \emph{Electron. J. Combin.}, \textbf{24} (2017), \#P3.12.

\bibitem{haglund2008q}
J.~Haglund, \emph{The {$q$},{$t$}-{C}atalan numbers and the space of diagonal harmonics}, vol.~41 of \emph{University Lecture Series}, American Mathematical Society, Providence, RI, 2008.

\bibitem{kazhdan1979representations}
D.~Kazhdan and G.~Lusztig, Representations of {C}oxeter groups and {H}ecke algebras, \emph{Invent. Math.}, \textbf{53} (1979), 165--184.

\bibitem{macdonald1980polynomial}
I.~G. Macdonald, Polynomial functors and wreath products, \emph{J. Pure Appl. Algebra}, \textbf{18} (1980), 173--204.

\bibitem{macdonald1995symmetric}
I.~G. Macdonald, \emph{Symmetric functions and {H}all polynomials}, Oxford Mathematical Monographs, The Clarendon Press, Oxford University Press, New York, 1995, 2nd ed.

\bibitem{matherne2023equivariant}
J.~P. Matherne, D.~Miyata, N.~Proudfoot, and E.~Ramos, Equivariant log concavity and representation stability, \emph{Int. Math. Res. Not. IMRN}, (2023), 3885--3906.


\bibitem{mendes2004lambdaring}
A.~Mendes, J.~Remmel, and J.~Wagner, A {$\lambda$}-ring {F}robenius characteristic for {$G\wr S_n$}, \emph{Electron. J. Combin.}, \textbf{11} (2004), \#R56.

\bibitem{proudfoot2019quniform}
N.~Proudfoot, Equivariant {K}azhdan--{L}usztig polynomials of {$q$}-niform matroids, \emph{Algebr. Comb.}, \textbf{2} (2019), 613--619.

\bibitem{proudfoot2021incidence}
N.~Proudfoot, Equivariant incidence algebras and equivariant {K}azhdan--{L}usztig--{S}tanley theory, \emph{Algebr. Comb.}, \textbf{4} (2021), 675--681.

\bibitem{proudfoot2016intersection}
N.~Proudfoot, M.~Wakefield, and B.~Young, Intersection cohomology of the symmetric reciprocal plane, \emph{J. Algebraic Combin.}, \textbf{43} (2016), 129--138.

\bibitem{proudfoot2018zpolynomial}
N.~Proudfoot, Y.~Xu, and B.~Young, The {$Z$}-polynomial of a matroid, \emph{Electron. J. Combin.}, \textbf{25} (2018), \#P1.26.

\bibitem{Speyer2026}
D.~E. Speyer, $L$-log-concavity and a proof of the conjecture of Lam, Postnikov and Pylyavskyy, \href{https://arxiv.org/abs/2601.05007v2}{arXiv:2601.05007v2}, 2026.

\bibitem{wu2023logconcavity}
S.~Wu and P.~B. Zhang, The log-concavity of {K}azhdan--{L}usztig polynomials of thagomizer matroids, \emph{Discrete Math.}, \textbf{346} (2023), \#113381.

\bibitem{xie2019thagomizer}
M.~H.~Y. Xie and P.~B. Zhang, Equivariant {K}azhdan--{L}usztig polynomials of thagomizer matroids, \emph{Proc. Amer. Math. Soc.}, \textbf{147} (2019), 4687--4695.

\end{thebibliography}

\appendix

\makeatletter
\let\savedseccntformat\@seccntformat
\renewcommand{\@seccntformat}[1]{\appendixname~\csname the#1\endcsname.\quad}
\section{The equivariant characteristic polynomial of \texorpdfstring{$T_n$ under $\Bn$}{Tn under Bn}}\label{app:charpoly}
\let\@seccntformat\savedseccntformat
\makeatother

We compute the equivariant characteristic polynomial of $T_n$ under the action of $\Bn$ and deduce strong induced log-concavity for the graded pieces of its Orlik--Solomon algebra.

\medskip

Let $M$ be a loopless matroid of rank $r$ with a finite group action $W\curvearrowright M$.
Write $\OS_M=\bigoplus_{i=0}^{r}\OS_{M,i}$ for its Orlik--Solomon algebra, viewed as a graded representation of $W$.
As in~\cite{gedeon2017equivariant}, define the \emph{equivariant characteristic polynomial} of $M$ under $W$ by
\[
\chi_M^{W}(t):=\sum_{i=0}^{r}(-1)^i\,t^{\,r-i}\,[\OS_{M,i}]\in \VRep(W)[t].
\]
\begin{prop}\label{prop:charpoly}
For all $n\ge 0$,
\[
\mathsf{\chi}_n(X,Y;t)
:=\chB\!\left(\chi_{T_n}^{\Bn}(t)\right)
=(t-1)\,h_n\!\big[(t-1)X-Y\big]
\in \Sym[X,Y][t].
\]
\end{prop}
\begin{proof}
By~\cite[Lemma~2.5]{gedeon2017equivariant}, every loopless
equivariant matroid $W\curvearrowright M$ satisfies
\[
t^{\rk M}\one_W
=\sum_{[F]\in\cL(M)/W}
\Ind_{W_F}^{W}\!\left(\chi_{M/F}^{W_F}(t)\right).
\]
For $T_n$, the Frobenius characteristic of the left-hand side is
\begin{equation}\label{eq:barzeta-thag}
t^{n+1}h_n[X]=t\,h_n[tX].
\end{equation}

For the Boolean matroid $B_m$ of rank $m\ge1$, let $\B{m}$ act through the quotient $\B{m}\to\Sg{m}$ by permuting its atoms. Inflating the equivariant formula for $\Sg{m}$~\cite[Lemma~3.3]{gao2022equivariant} gives
\begin{equation}\label{eq:boolean-charpoly}
\chB\!\left(\chi_{B_m}^{\B{m}}(t)\right)=h_m[(t-1)X].
\end{equation}
The formula also holds for $m=0$, since both sides equal $1$.

We split the orbit sum according to whether the representative flat contains $e_*$.

If $e_*\in F$, then $F$ is determined by a subset $S\subseteq\{1,\dots,n\}$ of size $k$, with stabilizer $\B{k}\times \B{n-k}$ (Lemma~\ref{lem:stabilizers}).
The contraction $T_n/F$ has a Boolean lattice of flats of rank $n-k$ (as in the proof of Proposition~\ref{prop:ZT-in-terms-of-PhiT}), with the remaining spikes as atoms.
The factor $\B{k}$ acts trivially on this lattice, and the action of $\B{n-k}$ factors through $\Sg{n-k}$.
The equivariant characteristic polynomial depends only on the group
action on the lattice of flats~\cite[Proposition~4.1]{proudfoot2021incidence}.
Equation~\eqref{eq:boolean-charpoly} gives $h_{n-k}[(t-1)X]$ for the
$\B{n-k}$ factor. Since $\B{k}$ acts trivially, induction to $\Bn$
and summation over $k$ give
\[
\sum_{k=0}^{n} h_k[X]\cdot h_{n-k}[(t-1)X]=h_n[tX]
\]
by the addition formula for complete homogeneous symmetric functions.

If $e_*\notin F$, then $F$ is determined by a subset of $k$ spikes together with a choice of one edge in each selected spike.
The action of $\B{k}$ on these choices affords the permutation representation $\Ind_{\Sg{k}}^{\B{k}}\one_{\Sg{k}}$, whose Frobenius characteristic is
\[
\chB\!\left(\Ind_{\Sg{k}}^{\B{k}}\one_{\Sg{k}}\right)=h_k[X+Y].
\]
For the standard representative $F_k^{\mathrm I}$, the contraction $T_n/F_k^{\mathrm I}$ is obtained from $T_{n-k}$ by replacing the spine with a parallel class of size $k+1$.
The natural isomorphism between its lattice of flats and that of $T_{n-k}$ commutes with the action of $\B{n-k}$, while the factor $\Sg{k}$ acts trivially on the lattice.
The $\B{n-k}$-equivariant characteristic polynomial of the contraction
therefore agrees with that of $T_{n-k}$. Inducing from
$\Sg{k}\times\B{n-k}$ to $\Bn$ and summing over $k$ gives
\[
\sum_{k=0}^{n} h_k[X+Y]\cdot \mathsf{\chi}_{n-k}(X,Y;t).
\]

Combining the two contributions with~\eqref{eq:barzeta-thag} gives the recursion
\[
t\,h_n[tX]=h_n[tX]+\sum_{k=0}^{n} h_k[X+Y]\cdot \mathsf{\chi}_{n-k}(X,Y;t).
\]
Since $H_{X+Y}(u)$ has constant term $1$, it is invertible as a formal power series.
Multiplying the recursion by $u^n$, summing over $n\ge0$, and dividing by
$H_{X+Y}(u)$ gives
\begin{align*}
\sum_{n\ge0}\mathsf{\chi}_n(X,Y;t)u^n
&=(t-1)\frac{H_X(tu)}{H_{X+Y}(u)}\\*
&=(t-1)\sum_{n\ge0}h_n[(t-1)X-Y]u^n.
\end{align*}
Comparing coefficients of $u^n$ proves the proposition.
\end{proof}

\begin{cor}\label{cor:os-strong-induced-logconcavity}
For every $n\ge0$, the sequence
$([\OS_{T_n,i}])_{i=0}^{n+1}$ is strongly induced log-concave with
respect to induction from $\Bn\times\Bn$ to $\B{2n}$.
\end{cor}

\begin{proof}
Fix $n\ge0$, and write $F_i=\chB([\OS_{T_n,i}])$, with $F_i=0$ outside
$0\le i\le n+1$. By~\eqref{eq:frob-bn-ind}, it suffices to prove
$F_iF_j-F_{i-1}F_{j+1}\succeq_s0$ for all $1\le i\le j$.
Proposition~\ref{prop:charpoly} gives
\[
\sum_{i=0}^{n+1}(-1)^i t^{n+1-i}F_i
=(t-1)h_n[tX-(X+Y)]
=(t-1)\sum_{k=0}^n(-1)^k t^{n-k}h_{n-k}[X]e_k[X+Y].
\]
Consequently,
\begin{equation}\label{eq:os-adjacent-sum}
F_i=G_i+G_{i-1},\qquad
G_i=h_{n-i}[X]e_i[X+Y]\quad(0\le i\le n),
\end{equation}
where $G_i=0$ otherwise. In particular, each $F_i$ is Schur-positive.
If $j\ge n+1$, then $F_{j+1}=0$, so the desired inequality follows.
It remains to consider $1\le i\le j\le n$.

We first show that the differences $D_{a,b}=G_aG_b-G_{a-1}G_{b+1}$
are Schur-positive for $0\le a\le b\le n$.
We use the conventions $h_{-1}=e_{-1}=0$ and $s_{\varnothing}=1$.
Adding and subtracting
$h_{n-a+1}[X]h_{n-b-1}[X]e_a[X+Y]e_b[X+Y]$
and applying the Jacobi--Trudi identities gives
\begin{align}
D_{a,b}
={}&s_{(n-a,n-b)}[X]e_a[X+Y]e_b[X+Y]\notag\\*
&+h_{n-a+1}[X]h_{n-b-1}[X]
s_{(2^a,1^{b-a})}[X+Y]\succeq_s0,
\label{eq:os-positive-difference}
\end{align}
where we have used
\[
h_Ah_B-h_{A+1}h_{B-1}=s_{(A,B)},\qquad
e_ae_b-e_{a-1}e_{b+1}=s_{(2^a,1^{b-a})}
\]
with $A=n-a\ge B=n-b\ge0$.
The identities also hold for $a=0$ and $b=n$.
Schur positivity follows from the Littlewood--Richardson rule,
applied both to products and to the expansion
\[
s_\lambda[X+Y]
=\sum_{\mu,\nu}c^\lambda_{\mu,\nu}s_\mu[X]s_\nu[Y].
\]

To pass from $(G_i)$ to $(F_i)$, we expand using~\eqref{eq:os-adjacent-sum}
and group the terms into differences $D_{a,b}$ with $a\le b$.
For $1\le i<j\le n$, this gives
\[
F_iF_j-F_{i-1}F_{j+1}
=D_{i,j}+D_{i-1,j-1}+D_{i,j-1}+D_{i-1,j}\succeq_s0.
\]
When $i=j$, the expansion gives
\[
F_i^2-F_{i-1}F_{i+1}
=D_{i,i}+D_{i-1,i-1}+D_{i-1,i}\succeq_s0.
\]
In both identities, the indices of each $D_{a,b}$ satisfy
$0\le a\le b\le n$, so~\eqref{eq:os-positive-difference} applies.
\end{proof}

\end{document}